\documentclass[hidelinks,onefignum,onetabnum]{siamart220329}

\usepackage{lipsum}
\usepackage{amsfonts}
\usepackage{graphicx}
\usepackage{epstopdf}
\usepackage{algorithmic}
\usepackage{mathtools}
\usepackage{amsmath}
\usepackage{amssymb}
\usepackage{tikz-cd}
\usepackage{bm}
\usepackage[cal=boondox]{mathalfa}
\newcommand{\rev}[1]{{\color{blue} #1}}

\newtheorem{assumption}[theorem]{Assumption}
\newtheorem{example}[theorem]{Example}
\ifpdf
  \DeclareGraphicsExtensions{.eps,.pdf,.png,.jpg}
\else
  \DeclareGraphicsExtensions{.eps}
\fi

\newsiamremark{remark}{Remark}
\newsiamremark{hypothesis}{Hypothesis}
\crefname{hypothesis}{Hypothesis}{Hypotheses}
\newsiamthm{claim}{Claim}
\newsiamthm{observation}{Observation}
\headers{MFEM for Linear Cosserat Equations}{W. M. Boon, O. Duran and J. M. Nordbotten}

\title{Mixed finite element methods for linear 
Cosserat equations\thanks{Submitted to the editors DATE.
\funding{European Research Council (ERC) under the European Union’s Horizon 2020 research and innovation program (grant agreement No 101002507)}}}

\author{W. M. Boon\thanks{MOX Laboratory, Department of Mathematics, Politecnico di Milano, Italy.}
\and O. Duran\thanks{Centre for Modeling of Coupled Subsurface Dynamics, Department of Mathematics, 
University of Bergen, Norway}
\and J. M. Nordbotten\footnotemark[3] \thanks{Norwegian Research Center, Postboks 22 Nygårdgaten, 5838, Bergen, Norway}}
 
\usepackage{amsopn}

\usepackage{enumitem}
\ifpdf
\hypersetup{
  pdftitle={Mixed finite element methods for linear Cosserat equations},
  pdfauthor={W. M. Boon, O. Duran and J. M. Nordbotten} 
}
\fi

\begin{document}

\maketitle

% REQUIRED
\begin{abstract}
    We consider the equilibrium equations for a linearized Cosserat material and provide two perspectives concerning well-posedness. First, the system can be viewed as the Hodge Laplace problem on a differential complex.
    % in terms of a differential complex, which is isomorphic to six copies of the de Rham complex through an algebraic isomorphism.
    On the other hand, we show how the Cosserat materials can be analyzed by inheriting results from linearized elasticity. Both perspectives give rise to mixed finite element methods, which we refer to as strongly and weakly coupled, respectively. We prove convergence of both classes of methods, with particular attention to improved convergence rate estimates, and stability in the limit of vanishing characteristic length of the micropolar structure. The theoretical results are fully reflected in the actual performance of the methods, as shown by the numerical verifications.  
\end{abstract}

% REQUIRED

% REQUIRED

\section{Introduction}
In this paper, we will study the equations for Cosserat materials in three spatial dimensions \cite{b1,b2}, also known as (micro-)polar materials among engineers \cite{b3}. Briefly speaking, these are an extension of the equations of linearized elasticity, by considering local rotations as an independent variable compared to the displacement. This additional generality afforded by Cosserat materials is attractive when modeling complex (and potentially multiscale) materials, such as granular media, cellular and crystalline solids, and composite porous media \cite{b4}.

Previous work on the numerical approximation of the Cosserat equations have been primarily considered with methods associated with the two-field formulation in terms of displacement and rotation. For an early application, see \cite{b5}. More recent contributions include \cite{zhou2015tetrahedral} which concerns tetrahedral finite elements and \cite{grbvcic2018variational,xie2019improving,b6} in which hexahedral elements are considered. We moreover mention the pseudo-spectral method proposed in \cite{tornabene2017mechanical} and the discrete element method of \cite{marazzato2021variational}. Our interest is in the formulation as a four-field system, wherein in addition to the aforementioned variables, also the mechanical stress and the so-called couple stress (associated with spatial variations in rotation) are explicitly represented in the discretization. In terms of regular elastic materials, it is well-known that such formulations with explicit stress variables avoid locking phenomena associated with nearly incompressible materials (see e.g.~\cite{b7,b8,b9}). Herein, we prove that in the context of Cosserat materials, the four-field formulation can be constructed so that it is robust also in the limit of vanishing characteristic length of the micropolar structure. This allows us to robustly model composite media consisting of elastic and Cosserat materials. 

Our contributions in this regards are as follows: 
\begin{enumerate}
 \item We use the observation from \cite[Sec. 5.2]{BGGSequences} that the Cosserat equations can be seen as the Hodge-Laplace equations on the \emph{Cosserat complex}. The $L^2$ version of the Cosserat complex provides the proper setting to prove well-posedness. 
 Following finite element exterior calculus, we propose a family of stable and convergent mixed finite element discretizations which we refer to as \emph{strongly coupled}. We moreover show that the discrete solution exhibits improved convergence beyond the standard a priori rates.
 \item	We also provide an alternative proof of well-posedness of the four-field formulation of the Cosserat equations using Ladyzhenskaya–Babuška–Brezzi (LBB) theory, through a relationship to the Hellinger–Reissner formulation of elasticity with so-called weak symmetry. This relationship shows that the Cosserat equations can be treated as a continuous generalization of the elasticity equations. 
 Using this perspective, we propose a family of \emph{weakly coupled} mixed finite element discretizations that are stable in the continuous transition between Cosserat and elastic materials.
 % \item	We show that the two perspectives on well-posedness stated in 1) and 2) naturally lead to different families of stable and convergent mixed finite element discretizations, which we refer to as \emph{strongly coupled} and \emph{weakly coupled}, respectively. Importantly, the weakly coupled discretization is stable in the continuous transition between Cosserat and elastic materials. 
\end{enumerate}

The rest of the paper is structured as follows. In the next section, we introduce the relevant notation and spaces, and state the linear Cosserat equations and their relationship to linearized elasticity. Thereafter, in Section \ref{sec3}, we elaborate the construction of the Cosserat complex, and the resulting discretization. In Section \ref{sec4}, we elaborate the well-posedness proof of the Cosserat equation by use of the LBB theory, and the resulting discretization. Sections \ref{sec3} and \ref{sec4} can be read independently of each other. In Section \ref{sec5}, the expected optimal-order convergence rates and stability properties are demonstrated in numerical examples. 

\section{Notation and problem definition}
\label{sec:main}

We first summarize the notation needed for the later developments, whereafter we state the equations of linear Cosserat materials. 

In the interest of space and clarity, we will exclusively detail the case of three spatial dimensions, wherein both displacement and rotation are characterized by three-dimensional vectors. The reduction to the simpler case of 2D problems is  obtained by considering line symmetry along the third dimension, whereafter the problem reduces to being characterized by a two-dimensional displacement vector, together with a single scalar rotation variable around the out-of-plane dimension. 

\subsection{General notation}

As a preliminary remark, we note that we will as far as possible adhere to the standard conventions of “vector-matrix calculus”, as opposed to either tensor or exterior calculus. This choice it justified by avoiding the proliferation of indices in tensor calculus, and by being accessible to a broader audience than exterior calculus. 

Throughout the paper, we will consider the domain $\Omega \subset \mathbb{R}^3$ to be open, bounded, and contractible, and we will consistently denote the spatial coordinate as $x\in \Omega$. As the domain will not change, we will often suppress the dependence on $\Omega$ in the statement of function spaces. That said, we denote spaces of vector fields by $\mathbb{V} \coloneqq \{v:\Omega\to \mathbb{R}^3 \}$, and spaces of matrix fields by $\mathbb{M} \coloneqq \{\sigma : \Omega \to \mathbb{R}^{3\times 3} \}$. To codify material constants, we will also need \rev{four-tensor fields} (i.e. mappings between matrices), the space of which we denote as  $\mathbb{A}=\{A : \mathbb{M}\to \mathbb{M} \}$. As indicated, we will in general use lower-case Latin letters for scalars and vectors, Greek letters for matrices, and capital Latin letters for fourth-order tensors. A notable exception to this convention is the parameterization indicated in Example \ref{ex2.4}, where $\mu$ and $\lambda$ are Lamé parameters, which is used throughout Section \ref{sec5}. 

Unless stated otherwise, we assume that the spaces have “sufficient” regularity for the relevant expressions to make sense. On the other hand, when the regularity of the fields needs to be treated explicitly, this is indicated by prefaces, e.g. $L^2 \mathbb{V}\coloneqq \{v\in \mathbb{V} : \int_\Omega v \cdot v \, dV < \infty \}$. %The subscript “0” on some space $V$, i.e. $V_0$, denotes the subspace of that space with compact support in $\Omega$ (i.e. that the function takes zero values on the boundary $\partial \Omega$). 

For a vector field $v\in \mathbb{V}$, with components $v_i$ and a matrix field $\sigma \in \mathbb{M}$ with components $\sigma_{i,j}$, we define the vector generalization of gradient and divergence, together with the \emph{asymmetry operator} $S : \mathbb{M} \to \mathbb{V}$, as
\begin{align}\label{eq2.1}
(\nabla v)_{i,j} &= \frac{\partial v_i}{\partial x_j}, 
& (\nabla \cdot \sigma)_i  &= \sum_j\frac{\partial \sigma_{i,j}}{\partial x_j}, 
& (S\sigma)_i &= \sigma_{i-1,i+1}-\sigma_{i+1,i-1}.
\end{align}
In the last definition, indices are understood modulo 3.  

By standard arguments \cite{b10}, we generalize the definition of the differential operators to their weak counterparts, which we still denote as $\nabla$ for the gradient and $\nabla\cdot$ for the divergence. These are densely defined in $L^2 \mathbb{V}$ and $L^2 \mathbb{M}$, respectively and their domains of definition are the vector extensions of $H^1 (\Omega)$ and $H(\nabla \cdot;\Omega)$, namely
\begin{align}
    H(\nabla; \mathbb{V}) &\coloneqq \left\{ v \in L^2 \mathbb{V} : \nabla v \in L^2 \mathbb{M}\right\}, &
    H(\nabla\cdot; \mathbb{M}) &\coloneqq \left\{ \sigma \in L^2 \mathbb{M} : \nabla \cdot \sigma \in L^2 \mathbb{V}\right\}\rev{.}
\end{align}

For all variables, we denote the standard $L^2$ inner product with parentheses, and the norm with double vertical strokes. As an example, for $\sigma,\tau \in L^2 \mathbb{M}$:
\begin{align}\label{eq2.2}
(\sigma, \tau)&\coloneqq\int_\Omega \sigma : \tau \, dV, & \|\sigma\|&\coloneqq(\sigma, \sigma)^{1/2},
\end{align}
where the matrix double dot product $:$ indicates the sum of element-wise products. 

As a convention, we will omit the subscript on the norm if the $L^2$ norm is used, and include subscripts whenever a different norm is considered. A second convention is that we will denote by the letter $0<C<\infty$ a generic computable constant. At various points, it will be emphasized within a concrete context what $C$ depends on. 

For conciseness, we will primarily consider homogeneous boundary conditions on the displacement and rotation fields, and we thus summarize the following adjoint, or integration by parts, relationships for $(v,\sigma) \in \mathbb{V} \times \mathbb{M}$ with $v = 0$ on $\partial \Omega$:
\begin{align}\label{eq2.3}
(\nabla \cdot \sigma , v)=(\sigma ,-\nabla v).
\end{align}

Similarly, we denote the adjoint (with respect to the standard $L^2$ inner product) of $S$ from \eqref{eq2.1} as $S^*$, which has the explicit form
\begin{align}\label{eq2.4}
S^*v=\begin{pmatrix}0&-v_3&v_2 \\ v_3&0&-v_1 \\ -v_2&v_1&0 \end{pmatrix}.
\end{align}

We note that in this case the adjoint is also (up to a factor 2) the right inverse of $S$, so that $SS^* v=2v$. Moreover, $S^*S \sigma = \sigma - \sigma^T$, which is double the ``$\operatorname{skew}$'' operator.

Finally, we will often consider a pair of vector or matrix spaces on $\Omega$. We denote these by script font, so that 
\begin{align}
    \mathcal{V} &\coloneqq \mathbb{V}\times\mathbb{V}
    & &\textrm{and} 
    & \mathcal{M} &\coloneqq \mathbb{M}\times\mathbb{M}.
\end{align}
All the previous definitions and conventions carry over to these compound spaces. 

For the exposition of mixed finite element spaces, we will need the definition of the relevant families of discrete subspaces for $H(\nabla\cdot;\mathbb{M})$ and $L^2(\Omega)$. Thus given a simplicial tesselation $\Omega_h$ of $\Omega$, we denote the standard mixed finite element spaces according to the following conventions \cite{b11}: 
\begin{itemize}
    \item The space $P_{-k}\left(\Omega_h\right)\subset L^2(\Omega)$ consists of $k$-order polynomials on each 3-simplex $\Delta\in\Omega_h$. 
    \item The space $P_k\left(\Omega_h\right)\coloneqq P_{-k}\Omega_h \cap C^0(\mathbb{R})$ consists of $k$-order polynomials on each 3-simplex $\Delta \in\Omega_h$ that are continuous across all element boundaries. 
	\item The Brezzi-Douglas-Marini \cite{brezzi1985two} space of order $k$ is defined as 
    \begin{align*}
        {BDM}_k\left(\Omega_h\right)\coloneqq P_{-k}\left(\Omega_h\right)^3\cap H(\nabla\cdot;\mathbb{M}).
    \end{align*}
	\item The Raviart-Thomas \cite{raviart2006mixed} space of order $k$ is defined as 
    \begin{align*}
    {RT}_k\left(\Omega_h\right)
    \coloneqq  \left(P_{-k}\left(\Omega_h\right)^3\oplus{\hat{P}}_{-k}\left(\Omega_h\right)\bm{x}\right)\cap H(\nabla\cdot;\mathbb{M}),
    \end{align*}
    where ${\hat{P}}_{-k}=P_{-k}\setminus P_{-\left(k-1\right)}$ are the homogeneous polynomials of exactly order $k$, and $\bm{x}$ is the position $x$ interpreted as a vector field.
\end{itemize}

We note that the convention for enumerating the $RT$-spaces varies in literature, here we adopt the most common convention where the lowest-order space is indexed as $RT_0$. %The notation introduced so far suffices for the bulk of the paper. Additional notation is needed for the introduction of the Cosserat complex in Section \ref{sec3}, which will be elaborated there. 

\subsection{The linearized Cosserat equations}
\label{sec2.1}

The Cosserat model of an elastic material concerns both (macroscopic) displacements $u$, as well as local (microscopic) rotations $r$ as primary state variables. These are associated with a linearized elastic stress field $\sigma$ and, respectively, a “couple stress” field $\omega$. In their linearized form, the model can be written as a set of four first-order equations in terms of the variables $\left(\sigma,\omega,u,r\right)\in\mathbb{M}\times\mathbb{M}\times\mathbb{V}\times\mathbb{V}$. The strong form of these equations is as follows:
\begin{subequations}\label{eq2.5}
\begin{align}
\textrm{Constitutive law for linearized elastic stress:}&
& A_\sigma\sigma &= \nabla u+S^\ast r, \label{eq2.5a}\\
\textrm{Constitutive law for linearized couple stress:}&
& A_\omega\omega &= \nabla r, \label{eq2.5b}\\
\textrm{Balance of linear momentum:}&
& -\nabla\cdot\sigma &= f_u, \label{eq2.5c}\\	
\textrm{Balance of angular momentum:}&
& -\nabla\cdot\omega+S\sigma &= f_r, \label{eq2.5d}			
\end{align}
    subject to the boundary conditions:
    \begin{align} \label{eq: bc}
        u &= 0, &
        r &= 0, &
        \textrm{on }& \partial \Omega.
    \end{align}

\end{subequations}

We see that the model equations depend on material tensors $A_\sigma$ and $A_\omega$. Based on physical arguments \cite{b12,b13,b4,b3}, we expect these to have the properties summarized in Assumptions \ref{a2.1}-2.3 below.

\begin{assumption}[Elastic material tensor]\label{a2.1}
    The elastic material tensor $A_\sigma\in\mathbb{A}$ is pointwise symmetric positive definite, and satisfies uniform upper and lower bounds, i.e. for any $\sigma,\tau\in L^2\mathbb{M}$ the following three statements hold with $A_\sigma^-, A_\sigma^+ \in \mathbb{R}$:
    \begin{align*}
        \left(A_\sigma\sigma,\tau\right) &= \left(A_\sigma\tau,\sigma\right),	
        & \frac{(A_\sigma \sigma,\sigma)}{\|\sigma\|^2} &\ge A_\sigma^->0, 
        & \frac{(A_\sigma \sigma,\tau)}{\|\sigma\|\|\tau\|} &\le A_\sigma^+<\infty.  
    \end{align*}
\end{assumption}

\begin{assumption}[Non-degenerate couple stress tensor]\label{a2.2a}
    The couple stress material tensor $A_\omega\in\mathbb{A}$ is pointwise symmetric positive definite, and satisfies uniform upper and lower bounds, i.e. for any $\omega,\rho\in L^2\mathbb{M}$ the following three statements hold with $A_\omega^-, A_\omega^+ \in \mathbb{R}$: 
    \begin{align*}
        \left(A_\omega\omega,\rho\right)&=\left(A_\omega\rho,\omega\right),	
        & \frac{(A_\omega \omega,\omega)}{\|\omega\|^2} &\ge A_\omega^->0, 
        & \frac{(A_\omega \omega,\rho)}{\|\omega\|\|\rho\|} &\le A_\omega^+<\infty.  
    \end{align*}
\end{assumption}

An important degeneracy is the limiting case when a Cosserat material reduces to a standard linearly elastic material, as we will detail in Sections \ref{sec2.3} and \ref{sec4}. This requires a weaker assumption on the material tensor for the couple stress: 

\begin{assumption}[Degenerate couple stress tensor]\label{a2.2b} The inverse of the couple stress material tensor $A_\omega$ can be decomposed as
    \begin{align*}
        A_\omega^{-1}=\ell^2{\tilde{A}}_\omega^{-1},
    \end{align*}
in which 
\begin{itemize}
\item $\ell\in P_1\left(\Theta_H\right)$ is a scalar function with $\Theta_H$ being a non-overlapping partitioning of the domain, satisfying (almost everywhere)
\begin{align*}
    0&\le\ell\le1, & \left|\nabla\ell\right|&\le C_\ell < \infty.
\end{align*}
\item ${\tilde{A}}_\omega \in \mathbb{A}$ is pointwise symmetric positive definite and satisfies uniform upper and lower bounds, i.e.~for any $\omega,\rho\in L^2\mathbb{M}$ the following three statements hold with $A_\omega^-, A_\omega^+ \in \mathbb{R}$: 
\begin{align*}
    ({\tilde{A}}_\omega\omega,\rho)&=({\tilde{A}}_\omega\rho,\omega),	
    & \frac{({\tilde{A}}_\omega \omega,\omega)}{\|\omega^2\|} &\ge A_\omega^->0, 
    & \frac{({\tilde{A}}_\omega \omega,\rho)}{\|\omega\|\|\rho\|} &\le A_\omega^+<\infty.   
\end{align*}
\end{itemize}
\end{assumption}

\begin{corollary}[Invertibility of material tensors]\label{cor2.3} Whenever Assumption \ref{a2.1} holds, $A_\sigma^{-1}$ is well defined, symmetric positive definite, and has upper and lower bounds (given by  $\left(A_\sigma^-\right)^{-1}$ and $\left(A_\sigma^+\right)^{-1}$, respectively). Similarly, if Assumption \ref{a2.2a} (or \ref{a2.2b}) holds, $A_\omega^{-1}$ (or ${\tilde{A}}_\omega^{-1}$), is symmetric positive semi-definite, with upper and lower bounds. 
\end{corollary}

We note that while Assumptions \ref{a2.1} and \ref{a2.2a} are sufficient for invertibility, they are not strictly necessary and may be relaxed to appropriate inf-sup conditions on $A_\omega$ or ${\tilde{A}}_\omega$. We make the slightly stronger assumptions since these are satisfied by the examples considered herein.

\begin{example}[Isotropic Cosserat material]\label{ex2.4}
    An isotropic material satisfies the condition that the material tensors $A_\sigma$ and $A_\omega$ are invariant under coordinate rotations. As is well-known (see e.g.~\cite{b14}), isotropic fourth-order tensors have only two free parameters when acting on symmetric matrices, known as the Lamé parameters in the context of elasticity. 
    To handle non-symmetric matrices common to Cosserat elasticity, we additionally require a Cosserat couple modulus \cite{b12}. 
    Denoting the elastic Lamé parameters as $\mu$ and $\lambda_\sigma$, and the Cosserat couple modulus by $\mu_\sigma^c$, we consider the elastic material tensor on $\sigma\in L^2\mathbb{M}$ as
    \begin{align}
        A_\sigma\sigma \coloneqq
        \frac{1}{2\mu}\left(\sigma-\frac{\lambda_\sigma}{2\mu+3\lambda_\sigma}\left(\sigma :I\right)I\right)
        + \left(\frac{1}{2\mu_\sigma^c} - \frac{1}{2\mu}\right) \frac12 S^*S \sigma\rev{.}
    \end{align}
    Similarly for the couple stress material tensor, we introduce the micropolar moduli $\lambda_\omega$ and $\mu_\omega^c$ and consider the following material tensor for $\omega\in L^2\mathbb{M}$
    \begin{align}
        % \tilde{A}_\omega\omega = 
        A_\omega\omega \coloneqq 
        \frac{1}{\ell^2} \left[
        \frac{1}{2\mu}\left(\omega-\frac{\lambda_\omega}{2\mu+3\lambda_\omega}\left(\omega : I\right)I\right) 
        + \left(\frac{1}{2\mu_\omega^c} - \frac{1}{2\mu}\right) \frac12 S^*S \omega
        \right],
    \end{align}
    in which $\ell$ is the characteristic length of the micropolar structure \cite{b12}.

    We remark that Assumption \ref{a2.1} holds if $\min\{\mu, \mu_\sigma^c\} \ge c >0$ and $2\mu+3\lambda_\sigma>0$, and similarly Assumption \ref{a2.2a} holds if $\min\{\ell, \mu, \mu_\omega^c\} \ge c > 0$ and $2\mu+3\lambda_\omega>0$. 
    Assumption \ref{a2.2b}, on the other hand, is satisfied in the more general case of $\ell \ge 0$ with $\min\{\mu, \mu_\omega^c\} \ge c >0$ and $2\mu+3\lambda_\omega>0$. 
    Importantly, we allow for $\lambda_\sigma=0$ or $\lambda_\omega=0$. 

    Furthering the example, subject to Assumption \ref{a2.1} the elastic material law can be inverted to obtain
    \begin{align}
        A_\sigma^{-1}\tau = 
        2 \mu  \left(I - \frac12 S^*S \right) \tau 
        + 2\mu_\sigma^c \left(\frac12 S^*S\right) \tau + \lambda_\sigma\left(\tau : I\right)I.
    \end{align}
    In the special case $\mu_\sigma^c = \mu$, this law simplifies to the familiar expression for isotropic and symmetric material tensors, $A_\sigma^{-1}\tau= 2 \mu \tau + \lambda_\sigma\left(\tau : I\right)I$.
\end{example}

\begin{remark}\label{r2.5}
    To model nearly incompressible materials, it is important to consider robustness with respect to the limit $\lambda_\sigma\rightarrow\infty$. This case is handled naturally in this work since we consider the mixed formulation of the elasticity equations. In particular, for an isotropic material as characterized in Example \ref{ex2.4}, the continuity and coercivity constants appearing in the proofs can be shown to be independent of $\lambda_\sigma$ due to \cite[Prop. 9.1.1]{b11}. We will therefore not consider the incompressible limit explicitly in the analyses of Sections \ref{sec3} and \ref{sec4}, but verify this property numerically in Section \ref{sec5.2}.
\end{remark}

\subsection{Connection to linear elasticity} \label{sec2.3}
An important aspect of the Cosserat equations, as well as the later analysis, is that they can be considered as a generalization of the equations of linearized elasticity. To see this, we first recall the Hellinger-Reissner form of the equations for elasticity. Let therefore as before $\left(\sigma,u,r\right)\in\mathbb{M}\times\mathbb{V}\times\mathbb{V}$ represent stress, displacement and rotation, then a static equilibrium of a linearly elastic system satisfies  \cite{b9,b11}: 
\begin{subequations}\label{eq2.6}
\begin{align}
\textrm{Constitutive law for linearized elastic stress:}&
&
A_\sigma \sigma &=\nabla u+S^\ast r,	\\
\textrm{Balance of linear momentum:} &
&
-\nabla\cdot\sigma &=f_u, \\
\textrm{Symmetry of the stress tensor:} &
&
S\sigma &=0,	
\end{align}

    subject to the boundary condition:
    \begin{align}
        u &= 0, &
        \textrm{on }& \partial \Omega.
    \end{align}

\end{subequations}

By inspection of equations \eqref{eq2.5} and \eqref{eq2.6}, we note that, formally, the equations of linearized elasticity can be obtained as the limit of the linearized Cosserat equations when the tensor $A_\omega \to \infty$ in a suitable sense (e.g. if $\ell \downarrow 0$ in Example \ref{ex2.4}), and thus $\omega\rightarrow 0$. This motivates the distinction highlighted by Assumptions \ref{a2.2a} and \ref{a2.2b}. 

This paper will not provide rigorous analysis of the relationship between equations \eqref{eq2.5} and \eqref{eq2.6}, but rather use the formal observation to motivate the need for numerical methods that are stable and convergent independent of an upper bound on $A_\omega$. This topic will be discussed in detail in Section \ref{sec4}. 

\begin{remark}
    The linear elasticity equations can alternatively be obtained by letting the Cosserat coupling modulus ($\mu_\sigma^c$ in Example \ref{ex2.4}) tend to zero. In that case, the material tensor $A_\sigma$ effectively enforces $S \sigma = 0$ strongly, leading to a symmetric elastic stress and a decoupling of the system. However, this limit violates Assumption \ref{a2.1} and we will therefore not consider it in this work. Instead, we focus on creating a method that is robust with respect to vanishing micropolar length scale $\ell \downarrow 0$.
\end{remark}

\subsection{Problem definition}
\label{sec2.4}
For future reference, we state the linearized Cosserat equations and highlight their structure. Consider first equations \eqref{eq2.5} written as a $4\times4$ block system: Find $\left(\sigma,\omega,u,r\right)\in\mathbb{M}\times\mathbb{M}\times\mathbb{V}\times\mathbb{V}$ such that 
\begin{align}\label{eq2.7}
    \begin{pmatrix}A_\sigma&&-\nabla&-S^\ast\\&A_\omega&&-\nabla\\-\nabla\cdot&&&\\S&-\nabla\cdot&&\\\end{pmatrix}\begin{pmatrix}
        \sigma \\ \omega \\ u \\ r
    \end{pmatrix}=
    \begin{pmatrix}
        0 \\ 0 \\ f_u \\ f_r
    \end{pmatrix},
\end{align}
subject to the boundary condition \eqref{eq: bc}. We highlight the structure of this system by introducing the block operators 
\begin{align}\label{eq2.8}
\mathcal{A} &\coloneqq \begin{pmatrix}A_\sigma&\\&A_\omega\\\end{pmatrix}, 
& \mathcal{S} &\coloneqq \begin{pmatrix}-\nabla\cdot&\\S&-\nabla\cdot\\\end{pmatrix}.
\end{align}
This allows us to rewrite the problem to: For $\mathcal{f}=\left(f_u,f_r\right)$, find $\left(\sigma,\omega\right)=\mathcal{p}\in\mathcal{M}$ and $\left(u,r\right)=\mathcal{u}\in\mathcal{V}$ such that
\begin{align}\label{eq2.9}
    \begin{pmatrix}\mathcal{A}&-\mathcal{S}^\ast\\\mathcal{S}&\\\end{pmatrix}
    \begin{pmatrix}\mathcal{p}\\ \mathcal{u}\end{pmatrix}=\begin{pmatrix}0 \\ \mathcal{f} \end{pmatrix}.
\end{align}

We emphasize the adjoint operator $\mathcal{S}^*$ is only defined for functions that satisfy zero boundary conditions in the sense of integration by parts, cf. \eqref{eq2.3} \cite{b10}, thus the weak formulation stated in equation \eqref{eq2.9} includes the boundary conditions naturally.   

Whenever both Assumptions \ref{a2.1} and \ref{a2.2a} hold, Corollary \ref{cor2.3} applies and the Schur-complement formulation of \eqref{eq2.9} is given by 
\begin{align}\label{eq2.10}
    \mathcal{S}\left(\mathcal{A}^{-1}\mathcal{S}^\ast\right)\mathcal{u}=\mathcal{f}.
\end{align}
After solving for $\mathcal{u}$, the variable $\mathcal{p}$ can be postprocessed as $\mathcal{p}=\mathcal{A}^{-1}\mathcal{S}^\ast\mathcal{u}$.

Our objective in the next sections is to provide well-posedness analysis and mixed finite element (MFE) methods for the system given in \eqref{eq2.9}, both in the context of Assumption \ref{a2.2a} (Section \ref{sec3}), as well as in the context of Assumption \ref{a2.2b} (Section \ref{sec4}).

\section{Strongly coupled mixed finite elements}
\label{sec3}
The Cosserat equations can be seen as a Hodge-Laplace operator on a differential complex \cite[Sec. 5.2]{BGGSequences}. This has the implication that general results from both analysis and MFE discretization of Hilbert complexes can be invoked. In particular, we obtain from general arguments both well-posedness and stability of both the continuous and discrete Cosserat equations, as well as quasi-optimal approximation properties of the resulting MFE solution.  

Our main concern are distributions with $L^2$ regularity, which leads to a Hilbert-Cosserat complex (for the extension of continuous complexes to the Hilbert setting, see \cite{b20,b17}). Thus, we consider the weak generalization of the calculus operators $\nabla$ and $\nabla\cdot$ to densely defined operators between distributions of $L^2$ regularity. In particular, we are interested in the mapping
\begin{align}\label{eq3.14}
\begin{tikzcd}[ampersand replacement=\&, column sep = 4.1em]
    H(\nabla\cdot;\mathcal{M}) \arrow[r, "{\begin{psmallmatrix}-\nabla\cdot & 0 \\ S & -\nabla\cdot \end{psmallmatrix}}"]
    \& L^2\mathcal{V}, %\arrow[r] 
\end{tikzcd}
\end{align}
where the domain of the block operator $\mathcal{S}$ is the space $H(\nabla\cdot;\mathcal{M})\coloneqq \{\mathcal{p}=\left(\sigma,\ \omega\right)\in L^2\mathcal{M} : \left(\nabla\cdot\sigma,\nabla\cdot\omega\right)\in L^2\mathcal{V}\}$.

We moreover introduce a weighted inner product on $L^2\mathcal{M}$ using the material tensors. For $\mathcal{r}=\left(\sigma,\omega\right), \mathcal{p}=\left(\tau,\rho\right)\in L^2\mathcal{M}$, we define
\begin{align}\label{eq3.4}
\left(\mathcal{r},\mathcal{p}\right)_\mathcal{A} \coloneqq \left(A_\sigma\sigma,\tau\right)+\left(A_\omega\omega,\rho\right),
\end{align}
with the associated norm $\|\mathcal{p}\|_\mathcal{A} \coloneqq (\mathcal{p},\mathcal{p})_\mathcal{A}^{1/2}$.

Then the following well-posedness of the weak (Hodge-)Laplacian is standard \cite[Thm 4.8]{b17}.

\begin{corollary}[Well-posedness with Dirichlet conditions]\label{cor3.5}
Let the mixed variational form of the Cosserat equations with Dirichlet conditions \eqref{eq2.10}, be stated as follows: For $\mathcal{f}\in L^2\mathcal{V}$, find $\left(\mathcal{p},\mathcal{u}\right)\in H(\nabla\cdot;\mathcal{M})\times L^2\mathcal{V}$ such that: 
\begin{subequations}\label{eq3.15}
    \begin{align}
        \left(\mathcal{p},\mathcal{p}^\prime\right)_\mathcal{A}-\left(\mathcal{u},\mathcal{Sp}^\prime\right) &= 0	
        & 	\forall\mathcal{p}^\prime &\in H(\nabla\cdot;\mathcal{M}),	\\	
        \left(\mathcal{Sp},\mathcal{u}^\prime\right) &= (\mathcal{f},\mathcal{u}^\prime)		
        & 	\forall\mathcal{u}^\prime &\in L^2\mathcal{V}.
    \end{align}
\end{subequations}
Then \eqref{eq3.15} admits a unique solution that satisfies
\begin{align}
    \|\mathcal{p}\|_\mathcal{A}+\|\mathcal{Sp}\| + \|\mathcal{u}\|\le C\|\mathcal{f}\|,
\end{align}
where the constant C depends on the Poincaré constant of the Cosserat complex, and thus scales with  $\sqrt{\max{\left(A_\sigma^+,A_\omega^+\right)}}$.
\end{corollary}

Our objective is to obtain a discrete approximation to the weak form of the first-order statement of the Cosserat equations. That is to say, for a family of finite-dimensional spaces $X_h^\mathcal{p}\subset H(\nabla\cdot;\mathcal{M})$ and $X_h^\mathcal{u}\subset L^2\mathcal{V}$, both indexed by $h$, we approximate equations \eqref{eq3.15} by the following finite-dimensional system:  For $\mathcal{f}\in L^2\mathcal{V}$, find $\left(\mathcal{p}_h,\mathcal{u}_h\right)\in X_h^\mathcal{p}\times X_h^\mathcal{u}$ such that: 
\begin{subequations}\label{eq3.17}
\begin{align}
\left(\mathcal{p}_h,\mathcal{p}_h^\prime\right)_\mathcal{A}-\left(\mathcal{u}_h,\mathcal{S}\mathcal{p}_h^\prime\right) &= 0 &
\forall \mathcal{p}_h^\prime &\in X_h^\mathcal{p}, \\
\left(\mathcal{S}\mathcal{p}_h,\mathcal{u}_h^\prime\right) &= (\mathcal{f},\mathcal{u}_h^\prime) &
\forall \mathcal{u}_h^\prime &\in X_h^\mathcal{u}.
\end{align}
\end{subequations}
We will in this section consider finite element pairs that satisfy the following definition:

\begin{definition}[Strongly coupled spaces]\label{def3.6} 
The finite element spaces $X_h^\mathcal{p}\times X_h^\mathcal{u}\subset H(\nabla\cdot;\mathcal{M})\times L^2\mathcal{V}$ form a strongly coupled pair if $\mathcal{S}X_h^\mathcal{p}\subseteq X_h^\mathcal{u}$ and there exist linear mappings $\pi^\star$ with $\star\in\left\{\mathcal{p},\mathcal{u}\right\}$ such that the following diagram  of operators arises: 
\begin{align}\label{eq3.18}
\begin{tikzcd}[ampersand replacement=\&]
H(\nabla\cdot;\mathcal{M}) 
\arrow[r, "\mathcal{S}"]
\arrow[d, "\pi^\mathcal{p}"]
\& L^2\mathcal{V}
% \arrow[r] 
\arrow[d, "\pi^\mathcal{u}"]
% \& 0 
\\ 
X_h^\mathcal{p}
\arrow[r, "\mathcal{S}"]
\& X_h^\mathcal{u}
% \arrow[r] 
% \&0
\end{tikzcd}
\end{align}
\end{definition}

\begin{remark}\label{r3.6b}
    For a strongly coupled pair of spaces, the condition $\mathcal{S}X_h^\mathcal{p}\subseteq X_h^\mathcal{u}$ implies that linear \eqref{eq2.5c} and angular \eqref{eq2.5d} momentum are balanced pointwise if $\mathcal{f}\in X_h^\mathcal{u}$. The nomenclature adopted here therefore relates to the strong or weak satisfaction of these balance equations.
\end{remark}

If the diagram \eqref{eq3.18} commutes, then the discrete spaces form a \emph{discrete subcomplex} of the relevant section of the Cosserat complex. The key for the next results is encoded in the continuity of the commuting linear maps $\pi^\star$ \cite{b11,b16}. Following those previous works, we recall the following definition. 

\begin{definition}[Bounded cochain projection]\label{def3.7} 
If the linear maps $\pi^\mathcal{p}$ and $\pi^\mathcal{u}$ satisfy the following properties: 
\begin{enumerate}
    \item 	Both operators are \emph{projections}, thus $\pi^\mathcal{p}\pi^\mathcal{p}=\pi^\mathcal{p}$ and $\pi^\mathcal{u}\pi^\mathcal{u}=\pi^\mathcal{u}$. 
    \item Both operators are \emph{uniformly bounded}, thus there exists a constant $C$, independent of $h$, such that 
    \begin{align*}
    \|\pi^\mathcal{p}\mathcal{p}\|_\mathcal{A}+\|\mathcal{S}\pi^\mathcal{p}\mathcal{p}\|  &\le C (\|\mathcal{p}\|_{\mathcal{A}}+\|\mathcal{Sp}\|)& \forall \mathcal{p} &\in H(\nabla\cdot;\mathcal{M}), \\
    \|\pi^\mathcal{u}\mathcal{u}\| &\le C\|\mathcal{u}\|& \forall \mathcal{u} &\in L^2\mathcal{V}.
\end{align*}
    \item The operators \emph{commute}: 
\begin{align*}
    \mathcal{S}\pi^\mathcal{u}=\pi^\mathcal{p}\mathcal{S}.
\end{align*}
\end{enumerate}
Then we refer to the two linear maps $\pi^\star$ collectively as a bounded cochain projection. 
\end{definition}

Existence of the above bounded cochain projections, which we will verify below, is sufficient to obtain unique solvability and optimal approximation properties. 

\begin{theorem}[Optimal approximation]\label{th3.8} 
Let the spaces $X_h^\mathcal{p}$ and $X_h^\mathcal{u}$ be strongly coupled, and let there exist bounded cochain projections $\pi^\star$. Then problem \eqref{eq3.17} is uniquely solvable and, moreover, the solution of \eqref{eq3.17} satisfies standard approximation properties relative to the solution of \eqref{eq3.15}: 
\begin{align}\label{3.19}
    \|\mathcal{p}-\mathcal{p}_h\|_\mathcal{A}+&\|\mathcal{Sp}-\mathcal{S}\mathcal{p}_h\|+\|\mathcal{u}-\mathcal{u}_h\| \nonumber\\ 
    &\le C\left(\inf_{\mathcal{p}^\prime\in H(\nabla\cdot;\mathcal{M})}(\|\mathcal{p}-\mathcal{p}^\prime\|_\mathcal{A}+\|\mathcal{S}(\mathcal{p}-\mathcal{p}^\prime)\|)+\inf_{\mathcal{u}^\prime\in L^2\mathcal{V}} (\|\mathcal{u}-\mathcal{u}^\prime\|) \right),
\end{align}
where the constant $C$ depends on the Poincar\'{e} constant of the section of the Cosserat complex given in Definition \ref{def3.6}, as well as the uniform bound $C$ on the bounded cochain projections stated in Definition \ref{def3.7}. 
\end{theorem}
\begin{proof}
    See \cite[Thm.~5.4 and 5.5]{b17}.
\end{proof}

\begin{theorem}[Strongly coupled MFE for Cosserat]\label{th3.9} 
For a family of simplicial partitions $\Omega_h$ of the domain $\Omega$, and for a non-negative integer $k$, define  
\begin{align}
    X_{h,k}^\mathcal{p} &\coloneqq W_k\left(\Omega_h\right)^3\times W_{k+1}\left(\Omega_h\right)^3
    &\textrm{with }W_k &\in \left\{BDM_{k+1},\ RT_k\right\},\\
    X_{h,k}^\mathcal{u} &\coloneqq P_{-k}\left(\Omega_h\right)^3\times P_{-\left(k+1\right)}\left(\Omega_h\right)^3.
\end{align}
Then the following statements hold:
\begin{enumerate}[label=\arabic*)]
    \item The spaces $X_{h,k}^\mathcal{p}$ and $X_{h,k}^\mathcal{u}$ are strongly coupled (Def. \ref{def3.6}). 
    \item There exist bounded cochain projections $\pi^\mathcal{p}$ and $\pi^\mathcal{u}$ (Def. \ref{def3.7}).
    \item If the solution of equation \eqref{eq3.15} has sufficient regularity, then the following \emph{a priori} convergence rates hold for $\mathcal{u}=\left(u,r\right)$ and $\mathcal{p}=\left(\sigma,\omega\right)$ 
    \begin{align}\label{eq3.20}
        \|\sigma-\sigma_h\|_{H(\nabla\cdot;\mathbb{M})}+\|u-u_h\|+\|\omega-\omega_h\|_{H(\nabla\cdot;\mathbb{M})}+\|r-r_h\|=\mathcal{O}\left(h^{k+1}\right).
    \end{align}
\end{enumerate}
\end{theorem}
\begin{proof}
    1) Clearly, $X_{h,k}^\mathcal{p}\subset H(\nabla\cdot;\mathcal{M})$ and $X_{h,k}^\mathcal{u}\subset L^2\mathcal{V}$, and any linear projection onto these spaces is a linear map. It remains to show that the range of $\mathcal{S}$ applied to $X_{h,k}^\mathcal{p}$ is contained in $X_{h,k}^\mathcal{u}$. Thus for any $\left(\sigma,\omega\right)\in X_{h,k}^\mathcal{p}$, we need to show that
    \begin{align}\label{eq3.21}
        \begin{pmatrix}-\nabla\cdot&0\\ S &-\nabla\cdot \end{pmatrix}\begin{pmatrix}
            \sigma \\ \omega
        \end{pmatrix}=\begin{pmatrix}
            -\nabla\cdot\sigma \\ S\sigma-\nabla\cdot\omega
        \end{pmatrix} \in \begin{pmatrix}
            P_{-k}\left(\Omega_h\right)^3 \\ P_{-\left(k+1\right)}\left(\Omega_h\right)^3
        \end{pmatrix}.
    \end{align}
By the definition of $X_{h,k}^\mathcal{p}$, we have that $\sigma\in W_k\left(\Omega_h\right)^3$ and $\omega\in W_{k+1}\left(\Omega_h\right)^3$, thus $\nabla\cdot\sigma\in P_{-k}\left(\Omega_h\right)^3$ and $\nabla\cdot\omega\in P_{-\left(k+1\right)}\left(\Omega_h\right)^3$ as differentiation reduces polynomial degree. It remains to show that $S\sigma\in P_{-\left(k+1\right)}\left(\Omega_h\right)^3$. Note that $W_k\left(\Omega_h\right)\subseteq P_{-\left(k+1\right)}\left(\Omega_h\right)^{3\times3}$, and since S is a linear transformation with constant weights, it does not increase polynomial degree. In turn, we have $S\sigma\in P_{-\left(k+1\right)}\left(\Omega_h\right)^3$, and thus the spaces are strongly coupled.

2) We first note that the existence of bounded cochain projections for the de Rham complex are well established, such as e.g. the Fortin operator and the $L^2$ projection (see e.g.~\cite[Prop. 2.5.2]{b11}). Let ${\tilde{\pi}}_{h, k}^\mathcal{p}$ and ${\tilde{\pi}}_{h, k}^\mathcal{u}$ be any such pair for the de Rham complex, i.e. such that
\begin{align}\label{eq3.22} 
\begin{tikzcd}[ampersand replacement=\&, column sep = 4.6em]
H(\nabla\cdot;\mathcal{M})
\arrow[r, "{\begin{psmallmatrix}-\nabla\cdot & 0 \\ 0 & -\nabla\cdot \end{psmallmatrix}}"] 
\arrow[d, "{\tilde{\pi}}_{h, k}^\mathcal{p}"]
\& L^2\mathcal{V} 
% \arrow[r] 
\arrow[d, "{\tilde{\pi}}_{h, k}^\mathcal{u}"]
% \&[-2em] 0 
\\ 
X_{h,k}^\mathcal{p} 
\arrow[r, "{\begin{psmallmatrix} -\nabla\cdot & 0 \\ 0 & -\nabla\cdot \end{psmallmatrix}}"] 
\& X_{h,k}^\mathcal{u} 
% \arrow[r] 
% \&0 
\end{tikzcd}
\end{align}
is a commuting diagram, and ${\tilde{\pi}}_{h, k}^\mathcal{p}$ and  ${\tilde{\pi}}_{h, k}^\mathcal{u}$ are uniformly bounded projections in $H(\nabla\cdot;\mathcal{M})$ and $L^2\mathcal{V}$, respectively. 

To obtain a bounded cochain projection for the Cosserat complex, we  will require the cross-product between the vector-interpretation of the spatial coordinate and vectors. Thus we define for $v\in\mathbb{V}$ (or similarly $v\in\mathbb{R}^3$) the element $Uv\in\mathbb{V}$ as:
\begin{subequations}\label{eq3.3}
\begin{align}
    \left(Uv\right)_i &= \bm{x}\times v=x_{i+1}v_{i-1}-x_{i-1}v_{i+1}. \label{eq3.3a}
\end{align}
We extend the definition to matrix fields, by letting the cross product act on the 
first index, i.e. 
\begin{align}
    \left(U\sigma\right)_{i,j} &= x_{i+1}\sigma_{i-1,j}-x_{i-1}\sigma_{i+1,j}.
\end{align}
\end{subequations}

The definition of $U$ allows us to introduce automorphisms for both $\mathcal{V}$ and $\mathcal{M}$ as follows:
% \begin{subequations}
\begin{align}\label{eq3.9}
    \Phi &= \begin{pmatrix}I&0\\U&I\end{pmatrix} &    &\textrm{with inverse} &     \Phi^{-1}&=\begin{pmatrix}I&0\\-U&I \end{pmatrix}.
\end{align}
% \end{subequations}

Moreover, we calculate the bound: 
\begin{align}\label{eq3.10}
    \|\Phi\|=\sup_{\mathcal{v}\in\mathcal{V}} \frac{\|\Phi\mathcal{v}\|}{\|\mathcal{v}\|} \le \sup_{\substack{\mathcal{v}\in\mathcal{V} \\ x\in\Omega}} \frac{(1+\|x\|)\|\mathcal{v}\|}{\|\mathcal{v}\|}=\sup_{x\in\Omega}(1+\|x\|)=C<\infty,
\end{align}
where $C$ depends only on the domain. By a similar calculation we find that also $\| \Phi^{-1} \| \le C$.

Since $\Phi$ is an isomorphism between the Cosserat complex and the de Rham complex \cite{BGGSequences}, we propose the following operators
\begin{align}\label{eq3.23}
    \pi^\mathcal{p} &\coloneqq \Phi{\tilde{\pi}}_{h, k}^\mathcal{p}\Phi^{-1}, &
    \pi^\mathcal{u} &\coloneqq \Phi{\tilde{\pi}}_{h, k}^\mathcal{u}\Phi^{-1}.
\end{align}
Following Definition 3.7, we first verify that these operators are projections. For $\pi^\mathcal{u}$, consider any $\left(u,r\right)\in X_{h,k}^\mathcal{u}=P_{-k}\left(\Omega_h\right)^3\times P_{-\left(k+1\right)}\left(\Omega_h\right)^3$, and recall that $U$ is a homogeneous linear function in $x$ (cf. \eqref{eq3.3a}). Therefore $\Phi^{-1}\left(u,r\right)=\left(u, r+Uu\right)\in P_{-k}\left(\Omega_h\right)^3\times P_{-\left(k+1\right)}\left(\Omega_h\right)^3=X_{h,k}^\mathcal{u}$. Now, since ${\tilde{\pi}}_{h, k}^\mathcal{u}$ is a projection onto $X_{h,k}^\mathcal{u}$, we derive 
\begin{align}\label{eq3.24}
    \pi^\mathcal{u} \begin{pmatrix}
        u \\ r
    \end{pmatrix} =\Phi\left({\tilde{\pi}}_{h, k}^\mathcal{u} \Phi^{-1} \begin{pmatrix}
        u \\ r
    \end{pmatrix} \right)= \Phi \left(\Phi^{-1} \begin{pmatrix}
        u \\ r
    \end{pmatrix} \right) = \begin{pmatrix}
        u \\ r
    \end{pmatrix},
\end{align}
and thus $\pi^\mathcal{u}$ is also a projection onto $X_{h,k}^\mathcal{u}$. 

 To show that also $\pi^\mathcal{p}$ is a projection, we start by showing that that $\Phi^{-1}X_{h,k}^\mathcal{p}\subseteq X_{h,k}^\mathcal{p}$.  We first consider the case $W_k=BDM_{k+1}$ and select any $\left(\sigma,\omega\right)\in{BDM}_{k+1}$ $\left(\Omega_h\right)^3\times{BDM}_{k+2}\left(\Omega_h\right)^3$.  Then also $\Phi^{-1}\left(u,r\right)=\left(\sigma,\omega+U\sigma\right)\in{BDM}_{k+1}\left(\Omega_h\right)^3\times{BDM}_{k+2} \left(\Omega_h\right)^3$, since the operator $U$ increases polynomial degree, but does not reduce regularity. 

Secondly for $W_k={RT}_k$, consider $\left(\sigma,\omega\right)\in{RT}_k\left(\Omega_h\right)^3\times{RT}_{k+1}\left(\Omega_h\right)^3$. Again we have $\Phi^{-1}\left(u,r\right)=\left(\sigma,\omega+U\sigma\right)\in{RT}_k\left(\Omega_h\right)^3\times{RT}_{k+1}\left(\Omega_h\right)^3$ since the operator $U$ is a linear combinations of products with coordinates, and by the definition of $RT_k$ spaces given in Section \ref{sec2.1}, it holds that for any coordinate $x_i$, and $w\in{RT}_k\left(\Omega_h\right)$ then $x_iw\in{RT}_{k+1}\left(\Omega_h\right)$. Finally, a similar calculation as in (3.24) completes the verification that $\pi^\mathcal{p}$ is a projection onto $X_{h,k}^\mathcal{p}$ for both cases of $W_k$.

From \eqref{eq3.10}, we know that the operators $\Phi$ and $\Phi^{-1}$ are bounded (independent of $h$ and $k$), and thus $\pi^\mathcal{p}$ and $\pi^\mathcal{u}$ inherit the boundedness of ${\tilde{\pi}}_{h, k}^\mathcal{p}$ and ${\tilde{\pi}}_{h, k}^\mathcal{u}$, respectively.

In order to verify commutativity, we first note that $\Phi$ allows us to explicitly rewrite
\begin{align}\label{eq3.25}
    \mathcal{S}=\Phi\begin{pmatrix}-\nabla\cdot & 0 \\ 0 & -\nabla\cdot \end{pmatrix} \Phi^{-1}.
\end{align}
Using this expression, we can easily verify the commutativity of the projection operators: 
\begin{align}\label{eq3.26}
\mathcal{S}\pi^\mathcal{p}=\Phi\begin{pmatrix}-\nabla\cdot&0\\0&-\nabla\cdot\\\end{pmatrix}\Phi^{-1}\Phi{\tilde{\pi}}_{h, k}^\mathcal{p}\Phi^{-1}=\Phi{\tilde{\pi}}_{h, k}^\mathcal{u}\begin{pmatrix}-\nabla\cdot&0\\0&-\nabla\cdot\\\end{pmatrix}\Phi^{-1}=\pi^\mathcal{u}\mathcal{S}.
\end{align}
We have thus established the existence of bounded cochain projections through an explicit construction.

3) is a corollary of Theorem \ref{th3.8}, and the error rates follow from the standard approximation properties of the spaces \cite{b11}, and the equivalence between the $\mathcal{A}$-weighted and unweighted norms on $L^2\mathcal{M}$ (a direct consequence of Assumptions \ref{a2.1} and \ref{a2.2a}). In particular, we have as a consequence of Theorem \ref{th3.8} that
\begin{align}
    \|\sigma-\sigma_h\|_{A_\sigma}&+\|\nabla\cdot\left(\sigma-\sigma_h\right)\|+\|u-u_h\|+\|\omega-\omega_h\|_{A_\omega} \nonumber\\ 
    &+\|S\left(\sigma-\sigma_h\right)-\nabla\cdot\left(\omega-\omega_h\right)\|+\|r-r_h\| =\mathcal{O}\left(h^{k+1}\right).\label{eq:3.38}
\end{align}
The stress norms satisfy
\begin{align*}
    \|\sigma-\sigma_h\|\le \left(A_\sigma^-\right)^{-\frac{1}{2}} \|\sigma-\sigma_h\|_{A_\sigma}  &=\mathcal{O}\left(h^{k+1}\right),  \\
     \|\omega-\omega_h\|\le \left(A_\omega^-\right)^{-\frac{1}{2}} \|\omega-\omega_h\|_{A_\omega} &=\mathcal{O}\left(h^{k+1}\right). 
\end{align*}
For the compound norm, we then calculate
\begin{align*}
\|\nabla\cdot\left(\omega-\omega_h\right)\|
&\leq \|S\left(\sigma-\sigma_h\right)-\nabla\cdot\left(\omega-\omega_h\right)\| + \|S\left(\sigma-\sigma_h\right)\|\\ 
&\leq \|S\left(\sigma-\sigma_h\right)-\nabla\cdot\left(\omega-\omega_h\right)\| + \|\sigma-\sigma_h\|\|S\|.
\end{align*}
Now, since $\|S\|$ is finite, it follows from equation \eqref{eq:3.38} that 
\begin{align}
    \|\nabla\cdot\left(\omega-\omega_h\right)\|\le \mathcal{O}\left(h^{k+1}\right),
\end{align}
from which \eqref{eq3.20} follows. 
\end{proof}

Higher convergence rates can to various degrees be obtained for all four variables. 
We present two different approaches that lead to complementary results in the following two theorems.

\begin{theorem}[Improved convergence for $BDM_{k+1}$] \label{th3.10}
Let $W_k = BDM_{k+1}$, and assume that the true solution has sufficient regularity. Then the
following improved convergence rates hold: 
\begin{align} \label{eq3.27}
 \left\|\sigma -\sigma _h\right\|+\left\|\omega -\omega
_h\right\|_{H\left(\nabla \cdot; \mathbb{M}\right)}+\left\|r-r_h\right\|=\mathcal{O}\left(h^{k+2}\right).
\end{align}
\end{theorem} 
\begin{proof}
We recall the improved $L^2$ estimates available for mixed methods \cite[Theorem 3.11]{arnold2010finite}:
\begin{align}
  \left\|\mathcal{p}-\mathcal{p}_h\right\| \leq C\left(\inf_{p^\prime \in X_h^{\mathcal{p}}} \left\|\mathcal{p}-\mathcal{p}^\prime\right\|+h \inf_{p^\prime \in X_h^{\mathcal{p}}}
  \left\|\mathcal{S}\left(\mathcal{p}-\mathcal{p}^\prime\right)\right\|\right). \label{eq328}
\end{align}

The approximation
properties of $X_h^{\mathcal{p}}$ and  $X_h^{\mathcal{u}}$ depend on the regularity of the solution,
and will be dominated by the lower-order spaces for displacement and stress. Concretely,  $\mathcal{p}_h \in X_{h, k}^{\mathcal{p}}$
implies that  $\sigma _h \in \mathit{BDM}_{k+1}\left(\Omega_h\right)^3$, thus we
obtain from \cite[Prop. 2.5.4]{b11} that 
\begin{align}
  \inf_{p^\prime \in X_h^{\mathcal{p}}} \left\|\mathcal{p}-\mathcal{p}^\prime\right\| &\leq Ch^{k+2}\left\|\sigma \right\|_{k+2}
  , &
  \inf_{p^\prime \in X_h^{\mathcal{p}}} \left\|\mathcal{S}\left(\mathcal{p}-\mathcal{p}^\prime\right)\right\| &\leq Ch^{k+1}\left\|\nabla \cdot \sigma
  \right\|_{k+1}. \label{eq329}
\end{align}

Thus for sufficiently regular  $\sigma $, we have from \eqref{eq328} and \eqref{eq329} that 
\begin{align}
   \left\|\sigma -\sigma _h\right\|+\left\|\omega -\omega
  _h\right\| = \mathcal{O}\left(h^{k+2}\right).
\end{align}

Consider now the error equation for the rotation and couple stress:
\begin{subequations} \label{eq330}
\begin{align} 
 \left(e_{\omega},\omega ^\prime\right)_\mathcal{A}-\left(e_r,\nabla \cdot \omega
^\prime\right) &= 0, \\
 \left(\nabla \cdot e_{\omega },r^\prime\right) &= \left(Se_{\sigma
},r^\prime\right).
\end{align}
\end{subequations}

Since \eqref{eq330} is a Laplace-type elliptic
equation, it follows from standard theory that the solution  $\left(e_{\omega },e_r\right)$ is bounded and satisfies 
\begin{align} \label{eq331}
 \left\|\omega -\omega _h\right\|_{H\left(\nabla \cdot ;M\right)}+\left\|r-r_h\right\| \leq C\left\|Se_{\sigma
}\right\|=\mathcal{O}\left(h^{k+2}\right),
\end{align}
where the last equality follows from \eqref{eq329}.
The theorem follows by combining \eqref{eq329} and \eqref{eq331}. 
\end{proof}

For the second analysis of improved convergence rates, applicable also to $W_k=RT_k$, we consider a duality argument. Let $\varpi_{h,k}$ denote the 
$L^2$-projection onto the piecewise polynomials $P_{-k}\left(\Omega_h\right)^3$, and define $\varphi =\left(\varphi _u,\varphi _r\right)$ as the solution to the Cosserat
equations \eqref{eq2.10} with right-hand side
\begin{align}
 f_u=\varpi_{h,k}\left(u-u_h\right) \ \ \ \text{and} \ \  f_r=\varpi_{h,k+1}\left(r-r_h\right).
\end{align}

Now, for sufficiently regular domains,  $\varphi$ will enjoy higher regularity than the minimum
regularity implied by the differential operators (see e.g.~\cite{grisvard2011elliptic}). Then the following theorem
applies.

\begin{theorem}[Improved convergence of rotation and displacement] \label{thm311}
Let
 $W_k \in \left\{BDM_{k+1},\mathit{RT}_k\right\}$, and  $\varphi $ as defined above be 
$H^2$ regular in the sense that
\begin{align} \label{ineq2}
    \left\|\mathcal{A}^{-1}\mathcal{S}^* \varphi \right\|_1
    + \left\|\varphi \right\|_2
    &\lesssim
    \left\|\varpi_h e_\mathcal{u}\right\|.
\end{align}
Then the rotation $r_h$ and displacement $u_h$ satisfy the optimal rates
\begin{align}
 \left\|r-r_h\right\|+\left\|\varpi_{h,k}u-u_h\right\| = \mathcal{O}\left(h^{k+2}\right).
\end{align}
\end{theorem}

\begin{proof}
The definition of  $\varphi $ implies that  $\mathcal{S}\mathcal{A}^{-1}\mathcal{S}^*\varphi =\varpi_h e_\mathcal{u}$ with 
$e_\mathcal{u}:=\left(e_u,e_r\right)=\left(u-u_h,r-r_h\right)$ and the $L^2$-projection $\varpi_h e_\mathcal{u}=\left(\varpi
_{h,k}e_u,\varpi_{h,k+1}e_r\right)$. Corollary~\ref{cor3.5} ensures that the solution is stable with 
\begin{align} \label{ineq1}
 \left\|\varphi \right\|+\left\|\mathcal{S}^* \varphi \right\| \leq C\left\|\varpi_h e_\mathcal{u}\right\|.
\end{align}

Introducing also  $e_\mathcal{p}:=\left(e_{\sigma },e_{\omega }\right)=\left(\sigma -\sigma _h,\omega -\omega _h\right)$, the
error equations for the Cosserat approximation are obtained by subtracting \eqref{eq3.17} from \eqref{eq3.15}: 
\begin{subequations}
\begin{align}
 \left(e_\mathcal{p},\mathcal{p}^\prime\right)_\mathcal{A}-\left(e_\mathcal{u},\mathcal{S}\mathcal{p}^\prime\right) &= 0, \label{eq350a} \\
 \left(\mathcal{S}e_\mathcal{p},\mathcal{u}^\prime\right) &= 0\rev{,} \label{eq350b}
\end{align}
\end{subequations}
which holds for all $\left(\mathcal{p}^\prime,\mathcal{u}^\prime\right) \in X_h^{\mathcal{p}}\times X_h^{\mathcal{u}}$.

Choosing now  $\mathcal{p}^\prime=\pi_h\mathcal{A}^{-1}\mathcal{S}^* \varphi $, this implies that  $\mathcal{S}\mathcal{p}^\prime=\varpi_h e_\mathcal{u}$, thus we obtain from
\eqref{eq350a} that
\begin{align} \label{eq351}
  \left\|\varpi_h e_\mathcal{u}\right\|^2 
  % \leq \left(\varpi_h e_\mathcal{u},\varpi_h e_\mathcal{u}\right)
  =\left(e_\mathcal{u},\varpi_h e_\mathcal{u}\right)
  =\left(e_\mathcal{u},\mathcal{S}\pi_h\mathcal{A}^{-1}\mathcal{S}^* \varphi \right)
  =\left(e_\mathcal{p},\pi_h\mathcal{A}^{-1}\mathcal{S}^* \varphi \right)_\mathcal{A}.
\end{align}

We now bound the right-hand side term:
\begin{align} \label{eq352}
    \left(e_\mathcal{p},\pi_h\mathcal{A}^{-1}\mathcal{S}^* \varphi \right)_\mathcal{A}
    &=-\left(e_\mathcal{p},\left(1-\pi_h\right)\mathcal{A}^{-1}\mathcal{S}^* \varphi\right)_\mathcal{A}
    +\left(e_\mathcal{p},\mathcal{A}^{-1}\mathcal{S}^* \varphi \right)_\mathcal{A} \nonumber\\
    &\lesssim  \|e_\mathcal{p}\| \|\left(1-\pi_h\right)\mathcal{A}^{-1}\mathcal{S}^* \varphi \|
    +\left(e_\mathcal{p},\mathcal{S}^* \varphi \right) \nonumber\\
    &\leq h \|e_\mathcal{p}\| \|\mathcal{A}^{-1}\mathcal{S}^* \varphi \|_1
    +\left(\mathcal{S}e_\mathcal{p},\left(1-\pi_h\right)\varphi
    \right)
    +\left(\mathcal{S}e_\mathcal{p},\pi_h\varphi \right) \nonumber\\
    &\leq h \|e_\mathcal{p}\| \|\mathcal{A}^{-1}\mathcal{S}^* \varphi \|_1
    +h^2 \|\mathcal{S}e_\mathcal{p}\| \|\varphi \|_2+0.
\end{align}

Here we have used equation \eqref{eq350b} to eliminate the last term, and moreover the fact that since  $\varphi $ is $H^2$-regular, then \cite{b11}:
\begin{align}
 \left\|\left(1-\pi_h\right)\varphi \right\| &\le h^2\left\|\varphi \right\|_2 
 & 
 \left\|\left(1-\pi
_h\right)\mathcal{A}^{-1}\mathcal{S}^* \varphi \right\| &\le h\left\|\mathcal{A}^{-1}\mathcal{S}^* \varphi \right\|_1.
\end{align}

Combining \eqref{eq351} and \eqref{eq352}, and using inequalities \eqref{ineq1} and \eqref{ineq2}, we obtain
\begin{align} \label{eq: e_r 1}
 \left\|\varpi_{h,k}e_u\right\|+\left\|\varpi_{h,k+1}e_r\right\| \leq C\left\|\varpi
_h e_\mathcal{u}\right\| = \mathcal{O}\left(h^{k+2}\right).
\end{align}

However, for the rotations, we also note that for sufficiently smooth solutions  $r$, it holds that:
\begin{align} \label{eq: e_r 2}
\left\|\varpi_{h,k+1}r-r\right\| = \mathcal{O}\left(h^{k+2}\right).
\end{align}

By combining \eqref{eq: e_r 1} and \eqref{eq: e_r 2}, we obtain
\begin{align*}
 \left\|e_r\right\|=\left\|r-\varpi_{h,k+1}r+\varpi_{h,k+1}r-r_h\right\| \leq \left\|\varpi
_{h,k+1}r-r\right\|+\left\|\varpi_{h,k+1}e_r\right\| = \mathcal{O}\left(h^{k+2}\right).
\end{align*}

This proves the theorem. 
\end{proof}

% \begin{remark}[Extensions]\label{r3.11} 
% The proofs of the above lemmas and theorems are based on rather general arguments. As such, we expect that it should be possible to reproduce the above constructions to obtain MFE discretizations of Hodge-Laplace equations for any part of the Cosserat complex. While we do not presently know of a physically relevant motivation for Hodge-Laplace equations on the $\mathcal{M}$-spaces of the complex, Corollary \ref{cor3.2} shows another manifestation of the Cosserat equations on the left side of the Cosserat complex. The counterpart of Corollary \ref{cor3.5} would be the variational system stated \textcolor{red}{in \eqref{eq3.16}}. The conforming finite element spaces for \textcolor{red}{equations \eqref{eq3.16}} in the context of the Cosserat complex is then to use $\mathcal{p}_h\in\mathbb{R}^6$ and standard Lagrange finite elements for $\mathcal{u}_h$, with the space of displacements having a polynomial degree one higher than the rotations. The analysis of two-field finite elements methods for Cosserat materials can be conducted by more direct arguments (standard Galerkin theory for quadratic minimization problems, see e.g.~\cite{b21}), and have been applied previously in literature \cite{b6}.      
% \end{remark}

\begin{remark}[Lack of stability with respect to $A_\omega^+$]\label{r3.12} Assumption \ref{a2.2a} ensures that $A_\omega^-$ and $A_\omega^+$ are bounded from below and above, respectively, and this is an essential component of the proof of the stability of the conforming MFE method defined in Theorem \ref{th3.9}. This has the practical implication that the conforming method of Theorem \ref{th3.9} will not be stable for degeneracy to normal elastic material laws (e.g. if $\ell \downarrow 0$ and thus $A_\omega^+\rightarrow\infty$). 
\end{remark}

\section{Weakly coupled mixed-finite elements}\label{sec4}
In this section, we wish to dispense with Assumption \ref{a2.2a}, which was a prerequisite for the development in Section \ref{sec3}. As such, we will dispense of the identification between the Cosserat equations with the Cosserat complex, since we can no longer expect $\left(\mathcal{p},\mathcal{p}^\prime\right)_\mathcal{A}$ to be a continuous bilinear form, and thus it fails to be an inner product on $L^2\mathcal{M}$.

Therefore, we will analyze the weak variational form of \eqref{eq2.10} directly, as stated in \eqref{eq3.15}. Thereafter, we will again consider MFE discretization as in \eqref{eq3.17}, however, this time we will have different conditions on the spaces, which results in a different discretization.

Throughout this section, both Assumptions \ref{a2.1} and \ref{a2.2b} will be assumed. Moreover, we emphasize that Remark \ref{r2.5} applies, that is to say that for isotropic materials, the constants appearing in the various proofs will be independent of the coefficient $\lambda_\sigma$.

\subsection{The degenerate linearized Cosserat equations}\label{sec4.1}
We begin by restating equations \eqref{eq2.7} subject to the substitution $A_\omega\rightarrow\ell^{-2}{\tilde{A}}_\omega$: Find $\left(\sigma,\omega,u,r\right)\in\mathbb{M}\times\mathbb{M}\times\mathbb{V}\times\mathbb{V}$ such that 

\begin{align}\label{eq4.1}
    \begin{pmatrix}A_\sigma&&-\nabla&-S^\ast\\&\ell^{-2}{\tilde{A}}_\omega&&-\nabla\\-\nabla\cdot&&&\\S&-\nabla\cdot&&\end{pmatrix}
    \begin{pmatrix}
        \sigma \\ \omega \\ u \\ r
    \end{pmatrix} = \begin{pmatrix}
        0 \\ 0 \\ f_u \\ f_r
    \end{pmatrix},
\end{align}
subject to the boundary conditions \eqref{eq: bc}. We recognize that the obstacle to showing well-posedness is the degeneracy of the diagonal block $\ell^{-2}{\tilde{A}}_\omega$ as $\ell \downarrow 0$, and therefore make the change of variables $\tilde{\omega}=\ell^{-1}\omega$, following the same strategy as \cite{arbogast2017mixed,boon2018robust}. We then obtain the problem: Find $\left(\sigma,\tilde{\omega},u,r\right)\in\mathbb{M}\times\mathbb{M}\times\mathbb{V}\times\mathbb{V}$ such that 
\begin{align} \label{eq4.2}
    \begin{pmatrix}A_\sigma&&-\nabla&-S^\ast\\&{\tilde{A}}_\omega&&-\ell\nabla\\-\nabla\cdot&&&\\S&-\nabla\cdot\ell&& \end{pmatrix} \begin{pmatrix}
        \sigma \\ \tilde{\omega} \\ u \\ r
    \end{pmatrix}
    =\begin{pmatrix}
        0 \\ 0 \\ f_u \\ f_r
    \end{pmatrix}.
\end{align}
Note that due to this change of variables, we are now able to consider the limit case of $\ell = 0$ in parts of the domain. This motivates defining the scaled block operators 
\begin{align}\label{eq4.3}
    \tilde{\mathcal{A}} &= \begin{pmatrix}A_\sigma&\\&{\tilde{A}}_\omega\\\end{pmatrix}, &
    \mathcal{S}_\ell &= \begin{pmatrix}-\nabla\cdot&\\S&-\nabla\cdot\ell\\\end{pmatrix}
\end{align}
and we can write for $\left(\sigma,\tilde{\omega}\right)=\tilde{\mathcal{p}}\in\mathcal{M}$ and $\left(u,r\right)=\mathcal{u}\in\mathcal{V}$:
\begin{align}\label{eq4.4}
    \begin{pmatrix}\tilde{\mathcal{A}}&-\mathcal{S}_{\ell}^\ast\\\mathcal{S}_\ell& \end{pmatrix}
    \begin{pmatrix}
        \tilde{\mathcal{p}} \\ \mathcal{u}
    \end{pmatrix} =
    \begin{pmatrix}
        0 \\ \mathcal{f}
    \end{pmatrix},
\end{align}
where $\mathcal{S}_{\ell}^\ast$ is the adjoint operator satisfying $\left(\mathcal{S}_{\ell}^\ast\mathcal{u},\mathcal{p}\right)=\left(\mathcal{u},\mathcal{S}_\ell\mathcal{p}\right)$ for all $\left(\mathcal{u},\mathcal{p}\right)\in\mathcal{V}\times\mathcal{M}$ with $\mathcal{u}$ satisfying \eqref{eq: bc}.

\subsection{Well-posedness of the degenerate linearized Cosserat equations}\label{sec4.2}
For the well-posedness analysis, we recognize that the space $H(\nabla\cdot;\mathbb{M})$ requires too much regularity of the couple stress when $\ell=0$. Thus, we begin by introducing the weighted space
\begin{align}\label{eq4.5}
    H(\nabla\cdot\ell;\mathbb{M}) \coloneqq \left\{\omega\in L^2\mathbb{M}\ \right|\ \nabla\cdot\left(\ell\omega\right)\in L^2\mathbb{V}\}.
\end{align}
Based on this space, we define the semi-weighted Hilbert space for $\mathcal{M}$, which we denote $H_\ell\mathcal{M}$ for simplicity, as
\begin{align}\label{eq4.6}
H_\ell\mathcal{M} \coloneqq H(\nabla\cdot;\mathbb{M})\times H(\nabla\cdot\ell;\mathbb{M}).
\end{align}
These spaces are naturally endowed with the norms:
\begin{align}
    \|\omega\|^2_{H(\nabla\cdot\ell)} &\coloneqq \|\omega\|^2+\|\nabla\cdot\left(\ell\omega\right)\|^2, & 
    \|\left(\sigma,\omega\right)\|^2_{H_\ell\mathcal{M}} &\coloneqq \|\sigma\|^2_{H(\nabla\cdot)}+\|\omega\|^2_{H(\nabla\cdot\ell)}.
\end{align}
With this choice of spaces, we establish sufficient conditions for well-posedness of \eqref{eq4.4} in the following three lemmas: 

\begin{lemma}[Continuity]\label{l4.1} 
The bilinear forms $\left(\tilde{\mathcal{A}}\mathcal{p},\mathcal{r}\right)$ and $\left(\mathcal{S}_\ell\mathcal{p},\mathcal{u}\right)$ are continuous with continuity constants independent of $\ell$.
\end{lemma}
\begin{proof}
    Let $\mathcal{p}=\left(\sigma,\ \omega\right)\in H_\ell\mathcal{M}$ and $\mathcal{r}=\left(\tau,\rho\right)\in H_\ell\mathcal{M}$. We first note the following trivial bound:
    \begin{align}
        \|\mathcal{p}\|^2_{H_\ell\mathcal{M}}=\|\sigma\|^2_{H(\nabla\cdot)}+\|\omega\|^2_{H(\nabla\cdot\ell)} \ge \|\sigma\|^2+\|\omega\|^2=\|\mathcal{p}\|^2.
    \end{align}
    From Assumptions \ref{a2.1} and \ref{a2.2b}, we now obtain the continuity:
    \begin{align}
        \left(\tilde{\mathcal{A}}\mathcal{p},\mathcal{r}\right)\le A_\sigma^+\|\sigma\|\|\tau\|+A_\omega^+\|\omega\|\|\rho\| 
        &\le \max\left(A_\sigma^+,A_\omega^+\right) \|\mathcal{p}\|\|\mathcal{r}\| \nonumber\\ 
        &\le\max{\left(A_\sigma^+,A_\omega^+\right)}\|\mathcal{p}\|_{H_\ell\mathcal{M}}\|\mathcal{r}\|_{H_\ell\mathcal{M}}.
\end{align}
    Similarly, for $\left(\mathcal{S}_\ell\mathcal{p},\mathcal{u}\right)$ with $\mathcal{p}=\left(\sigma,\ \omega\right)\in H_\ell\mathcal{M}$ and $\mathcal{u}=\left(u,r\right)\in L^2\mathcal{V}$, we obtain the following bound:
    \begin{align}
        \left(\mathcal{S}_\ell\mathcal{p},\mathcal{u}\right)\le\|\nabla\cdot\sigma\|\|u\|+\|S\sigma\|\|r\|+\|\nabla\cdot\left(\ell\omega\right)\|\|r\| \le C\|\mathcal{p}\|_{H_\ell\mathcal{M}}\|\mathcal{u}\|.
    \end{align}
\end{proof}

  \begin{lemma}[Coercivity]\label{l4.2} 
  The bilinear form $\left(\tilde{\mathcal{A}}\mathcal{p},\mathcal{r}\right)$ is coercive on the null-space of $\mathcal{S}_\ell$ in $H_\ell\mathcal{M}$, with a coercivity constant independent of $\ell$.    
  \end{lemma}
  \begin{proof}
      Consider any $\mathcal{p}=\left(\sigma,\omega\right)\in H_\ell\mathcal{M}$ with $\mathcal{S}_\ell\mathcal{p}=0$. Then both $\nabla\cdot\sigma=0$ and $\nabla\cdot\ell\omega=S\sigma$. Thus
\begin{align}
    \|\mathcal{p}\|^2_{H_\ell\mathcal{M}}=\|\sigma\|^2_{H(\nabla\cdot)}+\|\omega\|^2_{H(\nabla\cdot\ell)}=\|\sigma\|^2+\|\omega\|^2+\|S\sigma\|^2\le C\|\mathcal{p}\|^2,
\end{align}
    where the constant $C$ only depends on the bounded operator $S$. From Assumptions \ref{a2.1} and \ref{a2.2b}, we then obtain coercivity:
    \begin{align*}
        (\tilde{\mathcal{A}}\mathcal{p},\mathcal{p})\geq A_\sigma^-\|\sigma\|^2+A_\omega^-\|\omega\|^2\ge \min{(A_\sigma^-,A_\omega^-)} \|\mathcal{p}\|^2 \geq C^{-1} \min{(A_\sigma^-,A_\omega^-)} \|\mathcal{p}\|^2_{H_\ell\mathcal{M}}.
    \end{align*}
  \end{proof}
  
\begin{lemma}[LBB: inf-sup]\label{l4.3} The following inf-sup condition holds, with $\beta$ independent of $\ell$:
    \begin{align}\label{eq4.7}
        \inf_{\mathcal{u}\in L^2\mathcal{V}}\sup_{\mathcal{p}\in H_\ell\mathcal{M}} \frac{\left(\mathcal{S}_\ell\mathcal{p},\mathcal{u}\right)}{\|\mathcal{u}\|\|\mathcal{p}\|_{H_\ell\mathcal{M}}}\geq\beta>0.
    \end{align}
\end{lemma}
\begin{proof}
    We first spell out the inf-sup statement for $\mathcal{u}=\left(u,r\right)$ and $\mathcal{p}=\left(\sigma,\omega\right)$:
    \begin{align}\label{eq4.8}
        \inf_{\left(u,r\right)\in L^2\mathcal{V}} \sup_{\substack{\sigma\in H(\nabla\cdot;\mathbb{M}) \\ \omega\in H(\nabla\cdot\ell;\mathbb{M})}} \frac{-\left(\nabla\cdot\sigma,u\right)+\left(S\sigma,r\right)-\left(\nabla\cdot\left(\ell\omega\right),r\right)}{(\|u\|^2+\|r\|^2)^{\frac{1}{2}}(\|\sigma\|_{H(\nabla\cdot)}+\|\omega\|_{H(\nabla\cdot\ell)})^{\frac{1}{2}}}.
    \end{align}
    We can obtain a lower bound on expression \eqref{eq4.8} by considering in the supremum the subspace of $H_\ell\mathcal{M}$ where $\omega=0$, i.e. we have the bound
    \begin{align}\label{eq4.9}
        \inf_{\left(u,r\right)\in L^2\mathcal{V}} 
        \sup_{\substack{\sigma\in H(\nabla\cdot;\mathbb{M}) \\ \omega\in H(\nabla\cdot\ell;\mathbb{M})}} 
        &\frac{-\left(\nabla\cdot\sigma,u\right)+\left(S\sigma,r\right)-\left(\nabla\cdot\left(\ell\omega\right),r\right)}{(\|u\|^2+\|r\|^2)^{\frac{1}{2}}(\|\sigma\|_{H(\nabla\cdot)}+\|\omega\|_{H(\nabla\cdot\ell)})^{\frac{1}{2}}} \nonumber\\ 
        &\ge \inf_{\left(u,r\right)\in L^2\mathcal{V}} \sup_{\sigma\in H(\nabla\cdot;\mathbb{M})} \frac{-\left(\nabla\cdot\sigma,u\right)+\left(S\sigma,r\right)}{(\|u\|^2+\|r\|^2)^{\frac{1}{2}}\|\sigma\|_{H(\nabla\cdot)}}.
    \end{align}
    However, the right-hand side is known from the study of equations \eqref{eq2.6} to be bounded from below by a positive constant (see \cite{b9,b11}). Since $\ell$ does not appear in the right-hand side of \eqref{eq4.9}, the lower bound is independent of $\ell$, thus the lemma holds.
\end{proof}

\begin{corollary}[Well-posedness]\label{cor4.4} 
Let the weak mixed variational form of the Cosserat equations \eqref{eq4.4} be stated as follows: For $\mathcal{f}\in L^2\mathcal{V}$, find $\left(\tilde{\mathcal{p}},\mathcal{u}\right)\in H_\ell\mathcal{M}\times L^2\mathcal{V}$ such that: 
    \begin{subequations}\label{eq4.10}
        \begin{align}
            \left(\tilde{\mathcal{A}}\tilde{\mathcal{p}},\mathcal{p}^\prime\right)-\left(\mathcal{u},\mathcal{S}_\ell\mathcal{p}^\prime\right)&=0,\\
            \left(\mathcal{S}_\ell\tilde{\mathcal{p}},\mathcal{u}^\prime\right)&=(\mathcal{f},\mathcal{u}^\prime)\rev{,}
        \end{align}
    \end{subequations}
    for all $\left(\mathcal{p}^\prime,\mathcal{u}^\prime\right)\in H_\ell\mathcal{M}\times L^2\mathcal{V}$. Then equations \eqref{eq4.10} have a unique solution, satisfying
    \begin{align}\label{eq4.11}
        \|\tilde{\mathcal{p}}\|_{H_\ell\mathcal{M}}+\|\mathcal{u}\| \le C\|\mathcal{f}\|\rev{,}
    \end{align}
    \rev{where} the constant $C$ is finite and independent of $\ell$. 
\end{corollary}
\begin{proof}
    As a consequence of Lemmas \ref{l4.1}, \ref{l4.2}, and \ref{l4.3}, standard theory applies (see e.g.~\cite[Thm.~4.2.3]{b11}), and the constant $C$ only depends on the constants in the cited lemmas, which are independent of $\ell$.
\end{proof}

\subsection{Weakly coupled mixed finite element approximation}\label{sec4.3}
Our objective is now to obtain a finite-dimensional approximation to equations \eqref{eq4.10}. That is to say, for a family of finite-dimensional spaces  $X_h^\mathcal{p}\subset H(\nabla\cdot;\mathcal{M})$ and $X_h^\mathcal{u}\subset L^2\mathcal{V}$, both indexed by $h$, we approximate equations \eqref{eq4.10} by the following finite-dimensional system:  For $\mathcal{f}\in L^2\mathcal{V}$, find $\left({\tilde{\mathcal{p}}}_h,\mathcal{u}_h\right)\in X_h^\mathcal{p}\times X_h^\mathcal{u}$ such that: 
\begin{subequations}\label{eq4.12}
    \begin{align}
        \left(\tilde{\mathcal{A}}{\tilde{\mathcal{p}}}_h,\mathcal{p}_h^\prime\right)-\left(\mathcal{u}_h,\mathcal{S}_\ell\mathcal{p}_h^\prime\right)&=0	 
        & \forall \mathcal{p}_h^\prime &\in X_h^\mathcal{p}, \\
        \left(\mathcal{S}_\ell{\tilde{\mathcal{p}}}_h,\mathcal{u}_h^\prime\right)&=(\mathcal{f},\mathcal{u}_h^\prime)
        & \forall \mathcal{u}_h^\prime &\in X_h^\mathcal{u}.
    \end{align}
\end{subequations}

\rev{In this section, we will} assume that the mesh size $h$ is sufficiently small such that the triangulation $\Omega_h$ resolves $\Theta_H$ defined in Assumption \ref{a2.2b}. We then consider a pair of \emph{weakly coupled spaces}, in the sense that we will not require Definition \ref{def3.6} to hold. In particular, we will allow the range of $\mathcal{S}_\ell$ to be outside $X_h^\mathcal{u}$. Nevertheless, we will still consider finite-dimensional subspaces of $H(\nabla\cdot;\mathcal{M})$ and $L^2\mathcal{V}$, thus the discretization is conforming in the sense of function spaces. 

In the proof of Lemma \ref{l4.3} we exploited the relationship between the Cosserat equations and those of linearized elasticity, in the sense that despite the somewhat different structure, they share the same inf-sup constant. This observation motivates the following lemma: 

\begin{lemma}[Stability from elasticity]\label{l4.4} 
Let a quadruplet of finite-dimensional spaces $X_h^i,$ for $i\in\left\{\sigma, \omega, u, r\right\}$, be such that: 
   \begin{align}\label{eq4.13}
       X_h^\sigma &\subset H(\nabla\cdot;\mathbb{M}),  &
       X_h^\omega &\subset H(\nabla\cdot\ell;\mathbb{M}), &
       X_h^u &\subset L^2\mathbb{V}, &
       X_h^r &\subset L^2\mathbb{V},
   \end{align} 
   and let the spaces $X_h^\sigma,$ $X_h^u$ and $X_h^r$ satisfy the discrete LBB condition for linearized elasticity with weak symmetry: 
   \begin{align}\label{eq4.14}
       \inf_{\left(u,r\right)\in X_h^u\times X_h^r} \sup_{\sigma\in X_h^\sigma} \frac{-\left(\nabla\cdot\sigma,u\right)+\left(S\sigma,r\right)}{(\|u\|^2+\|r\|^2)^{\frac{1}{2}}\|\sigma\|_{H(\nabla\cdot)}} \geq\beta_h\geq\beta_0>0.
   \end{align}
   Then the following LBB condition holds for $X_h^\mathcal{u}\coloneqq X_h^\sigma\times X_h^\omega$ and $X_h^\mathcal{p} \coloneqq X_h^u\times X_h^r$:
   \begin{align}\label{eq4.15}
       \inf_{\mathcal{u}\in X_h^\mathcal{u}} \sup_{\mathcal{p}\in X_h^\mathcal{p}} \frac{\left(\mathcal{S}_\ell\mathcal{p},\mathcal{u}\right)}{\|\mathcal{u}\|\|\mathcal{p}\|_{H_\ell\mathcal{M}}}\geq\beta_0>0,
   \end{align}
   with $\beta_0$ independent of $\ell$.
\end{lemma}
\begin{proof}
    The same arguments as in the proof of Lemma \ref{l4.3} hold. 
\end{proof}

The significance of Lemma \ref{l4.4} is that any mixed finite element method for elasticity with weakly imposed symmetry will be applicable as a building block for a weakly coupled mixed finite element method for the Cosserat equations. The literature on these methods is reviewed in detail in \cite[Chapter 9]{b11}. Here, we will for concreteness consider the triplet proposed in \cite{b9}, consisting of $BDM$ elements for the stress, and equal-order discontinous Galerkin elements for displacement and rotation.

\begin{theorem}[Weakly coupled MFE for Cosserat equations]\label{th4.5} 
For a family of simplicial partitions $\Omega_h$ of the domain $\Omega$, and for a non-negative integer $k$, define  
\begin{subequations}
    \begin{align}
        X_{h,k}^\mathcal{p} &\coloneqq {BDM}_{k+1}\left(\Omega_h\right)^3\times W_k\left(\Omega_h\right)^3         
        &\textrm{with }W_k &\in \left\{BDM_{k+1},\ RT_k\right\}\rev{,}\\
        X_{h,k}^\mathcal{u} &\coloneqq P_{-k}\left(\Omega_h\right)^3\times P_{-k}\left(\Omega_h\right)^3,
    \end{align}
\end{subequations}    
and consider these spaces together with problem \eqref{eq4.12}. Then the following statements hold:
    \begin{enumerate}[label=\arabic*)]
        \item The stated choice of $X_h^\sigma, X_h^u$ and $X_h^r$ satisfies the assumptions of Lemma \ref{l4.4}.
        \item The resulting approximation is stable and convergent, independent of $\ell$, satisfying
        \begin{align*}
            \|\sigma-\sigma_h\|_{H(\nabla\cdot)}+ &\|\tilde{\omega}-{\tilde{\omega}}_h\|_{H(\nabla\cdot\ell)}+\|u-u_h\|+\|r-r_h\|  \\
            \le C&\left(\inf_{\sigma^\prime\in X_h^\sigma}\|\sigma-\tau\|_{H(\nabla\cdot)}  +\inf_{{\tilde{\omega}}^\prime\in X_h^\omega} \|\tilde{\omega}-{\tilde{\omega}}^\prime\|_{H(\nabla\cdot\ell)}\right. \\ 
            &\left. +\inf_{u^\prime\in X_h^u} \|u-u^\prime\|  +\inf_{r^\prime\in X_h^r} \|r-r^\prime \|\right).
        \end{align*}
        \item If the solution of \eqref{eq4.10} has sufficient regularity, then the following optimal \emph{a priori} convergence rates hold for $\mathcal{u}=\left(u,r\right)$ and $\mathcal{p}=\left(\sigma,\tilde{\omega}\right)$:
        \begin{subequations}\label{eq4.17}
            \begin{align}
                \|\sigma-\sigma_h\|_{H(\nabla\cdot)}+\|\tilde{\omega}-{\tilde{\omega}}_h\|_{H(\nabla\cdot\ell)}+\|u-u_h\|+\|r-r_h\| =\mathcal{O}\left(h^{k+1}\right). \label{eq4.17a}
            \end{align}
            In terms of non-scaled variables, this translates to:
            \begin{align}
                \|\sigma-\sigma_h\|_{H(\nabla\cdot)}+\|\ell^{-1}\left(\omega-\omega_h\right)\|+\|\nabla\cdot\left(\omega-\omega_h\right)\| & \nonumber \\ 
                +\|u-u_h\| +\|r-r_h\| &= \mathcal{O}\left(h^{k+1}\right). \label{eq4.17b}
            \end{align}
        \end{subequations}
    \end{enumerate}
\end{theorem}
\begin{proof}
    1) This is exactly \cite[Theorem 11.4]{b16}.

    2) We need in addition to the inf-sup condition from Lemma \ref{l4.4}, the coercivity and continuity of $(\tilde{\mathcal{A}}\mathcal{p},\mathcal{r})$ for this choice of spaces. This is established in Lemma \ref{l4.6} below, and thus again we can invoke \cite[Thm.~4.2.3 and 5.2.2]{b11}, which ensures stability and the approximation property \cite{b9}.

    3) This follows from 2), together with the approximation properties of the spaces \cite{b11}. We only need to show that the approximation properties of $X_h^\omega$ hold in the weighted norm $\| \cdot \|_{H(\nabla\cdot\ell)}$. However, a product rule and Assumption \ref{a2.2b} imply that this norm is weaker:
    \begin{align}
        \|\nabla \cdot \ell \tilde{\omega}\|
        \le \| \ell \nabla \cdot \tilde{\omega}\| + \|(\nabla \ell) \cdot  \tilde{\omega}\|
        \le C \| \tilde{\omega}\|_{H(\nabla \cdot)}.
    \end{align}
    In turn, the approximation properties of $X_h^\omega \subset H(\nabla\cdot; \mathbb{M})$ are directly inherited.
\end{proof}

The choice of weakly coupled spaces implies that the coercivity of $\tilde{\mathcal{A}}$ on the kernel of $\mathcal{S}_\ell$ is not an immediate consequence of Lemma \ref{l4.2}. We therefore show this in a separate lemma. For its proof, we require the following preliminary estimate.

\begin{lemma}\label{l4.7}
    Let $k\geq0$, $\ell$ satisfy Assumption \ref{a2.2b}, and let $\ell_h \in P_0(\Omega_h)$ be the maximum value of $\ell$ on each simplex $\Delta \in \Omega_h$. Then a constant $C_k$ exists such that:
    \begin{align}\label{eq4.24}
        \sup_{u\in P_{-k}\left(\Omega_h\right)^3} \frac{\|\ell_hu\|}{\|\sqrt{\ell\ell_h}u\|}\le C_k<\infty.
    \end{align}
    In particular, we have $C_0=2$ and $C_k=\sqrt5$ for $k\geq1$.
\end{lemma}
\begin{proof}
    To investigate this supremum, we take its square and rewrite as follows:
    \begin{align}\label{eq4.25}
       \sup_{u\in P_{-k}\left(\Omega_h\right)^3} \frac{\|\ell_hu\|^2}{\|\sqrt{\ell\ell_h}u\|^2} =\sup_{u\in P_{-k}\left(\Omega_h\right)^3} \frac{\int_\Omega \ell^2_h|u|^2 dV}{\int_\Omega {\ell\ell_h}|u|^2 dV}. 
    \end{align}
    Recall now that by Assumption \ref{a2.2b}, $\ell\in P_1\left(\Theta_H\right)$, and that $h$ is sufficiently small such that the triangulation $\Omega_h$ resolves $\Theta_H$. In turn, $\ell\in P_1\left(\Omega_h\right)$. Therefore, on a single 3-simplex $\Delta\in\Omega_h, \ell$ is a non-negative, linear function and the following mean value theorem applies
    \begin{align}\label{eq4.26}
        \frac{1}{4}\int_{\Delta}{\ell_hdV}\le\int_{\Delta}\ell dV.
    \end{align}
    Let $\varpi_k$ denote the $L^2$ projection onto $P_{-k}\left(\Omega_h\right)$. Using \eqref{eq4.26} with the fact that $\ell_h\in P_0\left(\Omega_h\right)$ and $\ell\in P_1\left(\Omega_h\right)$, we obtain
    \begin{align}\label{eq4.27}
        \frac{\int_\Omega \ell^2_h|u|^2 dV}{\int_\Omega {\ell\ell_h}|u|^2 dV} =\frac{\int_\Omega\ell_h^2\varpi_0(\left|u\right|^2)\ dV}{\int_\Omega \ell\ell_h\varpi_1(\left|u\right|^2)\ dV} \le 4 \frac{\int_{\Omega}{\ell\ell_h\varpi_0(\left|u\right|^2)\ dV}}{\int_{\Omega}{\ell\ell_h\varpi_1(\left|u\right|^2)\ dV}}.
    \end{align}
    
    For the case $k=0$, we have that $u$ is piecewise constant and thus $\varpi_1(\left|u\right|^2)=\varpi_0(\left|u\right|^2)=\left|u\right|^2$. Counteracting the square taken at the beginning of the proof, the combination of \eqref{eq4.25} and \eqref{eq4.27} gives us $C_0=\sqrt4=2$.

    We continue with $k\geq1$. Due to the projections $\varpi_k$, the supremum over $u\in P_{-k}\left(\Omega_h\right)^3$ can be rewritten as a supremum over $\phi\in P_{-1}\left(\Omega_h\right)$ with $\phi\geq0$:
    \begin{align}\label{eq4.28}
        \sup_{u\in P_{-k}\left(\Omega_h\right)} \frac{\int_{\Omega}{\ell\ell_h\varpi_0(\left|u\right|^2)\ dV}}{\int_{\Omega}{\ell\ell_h\varpi_1(\left|u\right|^2)\ dV}} 
        &= \sup_{\substack{\phi\in P_{-1}\left(\Omega_h\right) \\ \phi\geq0}} \frac{\int_{\Omega}{\ell\ell_h\varpi_0(\phi)\ dV}}{\int_{\Omega}{\ell\ell_h\phi\ dV}} \nonumber \\ 
        &= \sup_{\substack{\phi\in P_{-1}\left(\Omega_h\right) \\ \phi\geq0}} \frac{\int_{\Omega}{\varpi_0(\ell)\ell_h\varpi_0(\phi)\ dV}}{\int_{\Omega}{\ell\ell_h\phi\ dV}}.
    \end{align}
    Next, we consider a single simplex $\Delta\in\Omega_h$ and let $\ell_i$ denote the value of $\ell$ at vertex $v_i\in\Delta$. Using the fact that both $\ell$ and $\phi$ are nonnegative, linear functions, we use the definition of the local mass matrix of $P_1\left(\Delta\right)$ in 3D with exact integration (see e.g. \cite[Thm. 6.3.2]{van2023numerical}) to compute:
    \begin{align}\label{eq4.29}
        \int_{\mathrm{\Delta}}{\varpi_0(\ell)\varpi_0(\phi)\ dV}
        =\left|\mathrm{\Delta}\right|\frac{\sum_{i}\ell_i}{4}\frac{\sum_{i}\phi_i}{4} 
        &\le\frac{5}{4}\left(\left|\mathrm{\Delta}\right|\frac{\sum_{i}\ell_i\sum_{i}\phi_i+\sum_{i}{\ell_i\phi_i}}{20}\right) \nonumber\\ 
        &=\frac{5}{4}\int_{\mathrm{\Delta}}{\ell\phi\ dV}.
    \end{align}
    Scaling both sides by $\ell_h$ and summing over all $\Delta$, we combine \eqref{eq4.29} with \eqref{eq4.25}, \eqref{eq4.27}, and \eqref{eq4.28} to obtain the upper bound:
    \begin{align}
        \sup_{u\in P_{-k}\left(\Omega_h\right)^3} \frac{\|\ell_hu\|^2}{\|\sqrt{\ell\ell_h}u\|^2}\le 4 \sup_{\substack{\phi\in P_{-1}\left(\Omega_h\right) \\ \phi\geq0}} \frac{\int_{\Omega}{\varpi_0(\ell)\ell_h\varpi_0(\phi)\ dV}}{\int_{\Omega}{\ell\ell_h\phi\ dV}}\le4\left(\frac{5}{4}\right)=C_k^2.
    \end{align}
    In turn, we obtain $C_k=\sqrt5$ for $k\geq1$.
\end{proof}

The estimate from Lemma~\ref{l4.7} allows us to prove the following result.

\begin{lemma}[Discrete coercivity]\label{l4.6} 
With the choice of finite-dimensional subspaces $X_{h,k}^\mathcal{p}\subset H_\ell\mathcal{M}$ and $X_{h,k}^\mathcal{u}\subset L^2\mathcal{V}$ given in Theorem~\ref{th4.5}, the bilinear form $\tilde{\mathcal{A}}$ is coercive on the null-space $\mathcal{S}_\ell$ on $X_{h,k}^\mathcal{u}$, with coercivity constant independent of $\ell$.
\end{lemma}
\begin{proof}
    Let $\mathcal{p}=\left(\sigma,\omega\right)\in X_{h,k}^\mathcal{p}$ such that
    \begin{align}\label{eq4.18}
        \left(\mathcal{S}_\ell\mathcal{p},\mathcal{u}^\prime\right)&=0 &
        \forall \mathcal{u}^\prime &\in X_{h,k}^\mathcal{u}.
    \end{align}
We now note that $\nabla\cdot X_{h, k}^\sigma\subseteq X_{h,k}^u$, so that we may set $\mathcal{u}^\prime=\left(\nabla\cdot\sigma,0\right)$ in \eqref{eq4.18} and conclude $\nabla\cdot\sigma=0$. Secondly, we set $\mathcal{u}^\prime=\left(0,\ell_h\nabla\cdot \omega\ \right)$ in \eqref{eq4.18}, where $\ell_h\in P_0\left(\Omega_h\right)$ is a constant on each element equal to the maximum value of $\ell$ on that element. Thus $\ell_h\nabla\cdot\omega\in X_{h,k}^r$, and \eqref{eq4.18} evaluates to: 
\begin{align}\label{eq4.19}
    \left(S\sigma+\nabla\cdot\left(\ell\omega\right), \ell_h\nabla\cdot\omega\right) = 0.
\end{align}
Due to Assumption \ref{a2.2b}, we have sufficient regularity to apply the chain rule, and thus \eqref{eq4.19} is rewritten as
\begin{align}\label{eq4.20}
    \left(\ell\nabla\cdot\omega,\ell_h\nabla\cdot\omega\right)+\left(S\sigma+\omega\nabla\ell,\ell_h\nabla\cdot\omega\right)=0. 
\end{align}
We now calculate
\begin{align}\label{eq4.21}
    \|\sqrt{\ell\ell_h}\nabla\cdot\omega\|^2
    =\left(\ell\nabla\cdot\omega,\ell_h\nabla\cdot\omega\right)
    &=-\left(S\sigma+\omega\nabla\ell,\ell_h\nabla\cdot\omega\right) \nonumber\\ 
    &\le \|S\sigma+\omega\nabla\ell\| \|\ell_h\nabla\cdot\omega\|.
\end{align}
To proceed, we note that $\nabla\cdot\omega\in P_{-k}\left(\Omega_h\right)^3$. Thus Lemma \ref{l4.7} ensures that a finite $C_k$ exists with which we rewrite \eqref{eq4.21} as 
\begin{align}\label{eq4.22}
    \|\ell_h\nabla\cdot\omega\| 
    &\le C_k \|\sqrt{\ell\ell_h}\nabla\cdot\omega\| 
    \le C_k \|S\sigma+\omega\nabla\ell\| \frac{\|\ell_h\nabla \cdot \omega\|}{\|\sqrt{\ell\ell_h}\nabla \cdot \omega\|}
    \le C_k^2 \|S\sigma+\omega\nabla\ell\|. 
\end{align}
We now establish the following bound for $\mathcal{p}=\left(\sigma,\omega\right)\in X_{h,k}^\mathcal{p}$ satisfying \eqref{eq4.18}:
\begin{align}
    \left(\|\mathcal{p}\|^2_{H_\ell\mathcal{M}}-\|\mathcal{p}\|^2\right)^{\frac{1}{2}}
    &=\|\omega\nabla\ell+\ell\nabla\cdot\omega\| \nonumber \\
    &\le \|\omega\nabla\ell\|+\|\ell\nabla\cdot\omega\| \nonumber \\ 
    &\le \|\omega\nabla\ell\|+\|\ell_h\nabla\cdot\omega\| \nonumber \\
    &\le \|\omega\nabla\ell\|+C_k^2\|S\sigma+\omega\nabla\ell\| \nonumber \\
    &\le\left(1+C_k^2\right)\|\omega\nabla\ell\|+C_k^2\|S\sigma\| \nonumber \\
    &\le\left(1+C_k^2\right)C_\ell\|\omega\|+C_k^2\|S\|\|\sigma\| \nonumber \\ 
    &\le \max \left(\left(1+C_k^2\right)C_\ell,C_k^2 \|S\|\right) \sqrt{2}\|\mathcal{p}\|.
\end{align}
Squaring both sides leads us to
\begin{align}\label{eq4.23}
    \|\mathcal{p}\|^2_{H_\ell\mathcal{M}}\le C\|\mathcal{p}\|^2,
\end{align}
where the constant $C$ does not depend on $\ell$, but only on the assumed upper bound on its gradient $C_\ell$.  Coercivity now follows by the same calculation as in Lemma \ref{l4.2}. 
\end{proof}

\begin{remark}[Improved convergence]\label{r4.8.0} It is natural to ask whether improved convergence results similar to Theorem \ref{th3.10}  and \ref{thm311} can be developed for the weakly coupled discretizations. Inspecting the proofs of these theorems, it is clear that the improved convergence results can not in general be expected to be parameter robust with respect to $\ell$ for the weakly coupled method. We will therefore not detail improved convergence results in this section, but summarize the observed improved convergence rates in Section \ref{sec5}. 
\end{remark} 

\begin{remark}[Computational cost]\label{r4.8} The weakly coupled discretization given in Theorem \ref{th4.5} is nominally a cheaper method than the strongly coupled discretization in Theorem \ref{th3.9} for a given element degree $k$. This is due to the fact that lower-order spaces are employed for the rotation variables, thus more evenly balancing the computational cost between displacements and rotations, leading to a lower total number of degrees of freedom. On the other hand, the strongly coupled $BDM$ method enjoys some form of improved convergence rates of all variables, and is thus comparatively speaking a $k+1$-order method. The choice of most efficient discretization for a given problem is thus a balance of considerations. 
\end{remark}

\begin{remark}[Elasticity]\label{r4.9} Finally, we emphasize that by construction, Theorem \ref{th4.5} contains a stable discretization of linear elasticity, coinciding with the so-called Arnold-Falk-Winther elements \cite{b9}. Thus, if $\ell=0$ on any nondegenerate subdomain $\Omega_0\subseteq\Omega$, then for any choice of finite-dimensional sub-space $X_h^\omega$ the discrete solution is ${\tilde\omega}_h=0$ on $\Omega_0$. Moreover, the remaining variables still satisfy all properties shown in Theorem \ref{th4.5}. Finally, the elastic discretization is robust in the limit of incompressible materials \cite{b9,b11}, a property that is inherited by the method when applied to Cosserat materials.  
\end{remark}

\section{Numerical verification}\label{sec5}
We verify the performance of the MFE discretizations proposed above in terms of convergence rates and robustness. We consider three examples. In the first and second examples, we assess the convergence and robustness relative to the material properties of the couple stress and linear stress, respectively. In the third example, we consider a choice of material tensors such that a pure linear elastic medium is modeled in parts of the domain, thus verifying the applicability of the weakly coupled MFE discretization. 

Results are presented for the discretization with strong coupling detailed in Theorem \ref{th3.9}, with both choices of couple stress space, i.e.~$W_k=RT_k$ and $W_k=BDM_{k+1}$. We abbreviate this discretization as “SC-RT” and “SC-BDM”. Results are also presented for the discretization with weak coupling detailed in Theorem \ref{th4.5}, also with both choices of couple-stress space, which are abbreviated as “WC-RT” and “WC-BDM”. 

Throughout the section, we will consider an isotropic medium according to Example \ref{ex2.4}, with material parameters: 
\begin{align}\label{eq5.1}
    \mu =\lambda_\omega  &= 1,	&	
    \mu_\sigma^c = \mu_\omega^c &= 1/10.
\end{align}
Note that the values of the coupling coefficients $\mu_\sigma^c$ and $\mu_\omega^c$ are chosen as non-unitary, to avoid the simplification mentioned at the end of Example \ref{ex2.4}. The values of $\lambda_\sigma$ and $\ell$ will be varied to validate robustness, and are thus specified in the various subsections. Nevertheless, we note that for $0\le\lambda_\sigma<\infty$, these material parameters clearly satisfy Assumption \ref{a2.1}, and further satisfy Assumption \ref{a2.2a} if $\ell>0$. Finally, for spatially varying (piecewise linear) $\ell\in\left[0,1\right]$, Assumption \ref{a2.2b} holds.

When assessing the convergence of the methods, the norms given in Theorems \ref{th3.9} and \ref{th4.5} are used for the strongly and weakly coupled discretizations, respectively. When assessing the improved convergence of the methods, we use method-specific norms, consolidating both the theoretical results (Theorems \ref{th3.10} and \ref{thm311}) as well as the numerical results. For the strongly coupled methods, the norms in which we measure improved convergence are given by: 
\begin{subequations} \label{eq5.IC-norms1}
\begin{align} 
    \left\|(\mathcal{p},\mathcal{u} )\right\|^*_\mathrm{SC-RT} &= 
    \left\|\omega\right\| 
    +\left\{\left\|r\right\|
    +\left\|\varpi_{h,k}u\right\|\right\} 
    \label{eq5.IC-norms1a}\\
    \left\|(\mathcal{p} ,\mathcal{u} )\right\|^*_\mathrm{SC-BDM} &= 
    \left\{\left\|\sigma \right\|
    +\left\|\omega \right\|_{H\left(\nabla \cdot; \mathbb{M}\right)}
    +\left\|r\right\|+\left\|\varpi_{h,k}u\right\|\right\}
    \label{eq5.IC-norms1b}
\end{align}
\end{subequations}

Here the braces around terms in the norm indicates that improved convergence is proved in the cited theorems, while for the remaining term improved convergence is seen numerically, but has not been proved. For the weakly coupled methods, the improved convergence is not in general $\ell$-robust, and we therefore specify the norms of improved convergence in the relevant sections. 

For all cases, we consider as domain a unit cube domain, with manufactured solutions in $C_0^\infty\mathcal{V}$. In Sections \ref{sec5.1} and \ref{sec5.2}, a regular triangulation is employed as illustrated in Figure \ref{fig5.1}. In Section \ref{sec5.3}, the triangulation is adapted to resolve the structure of $\ell(x)$.

\begin{figure}[ht!]
    \centering
    \includegraphics[width=0.48\textwidth]{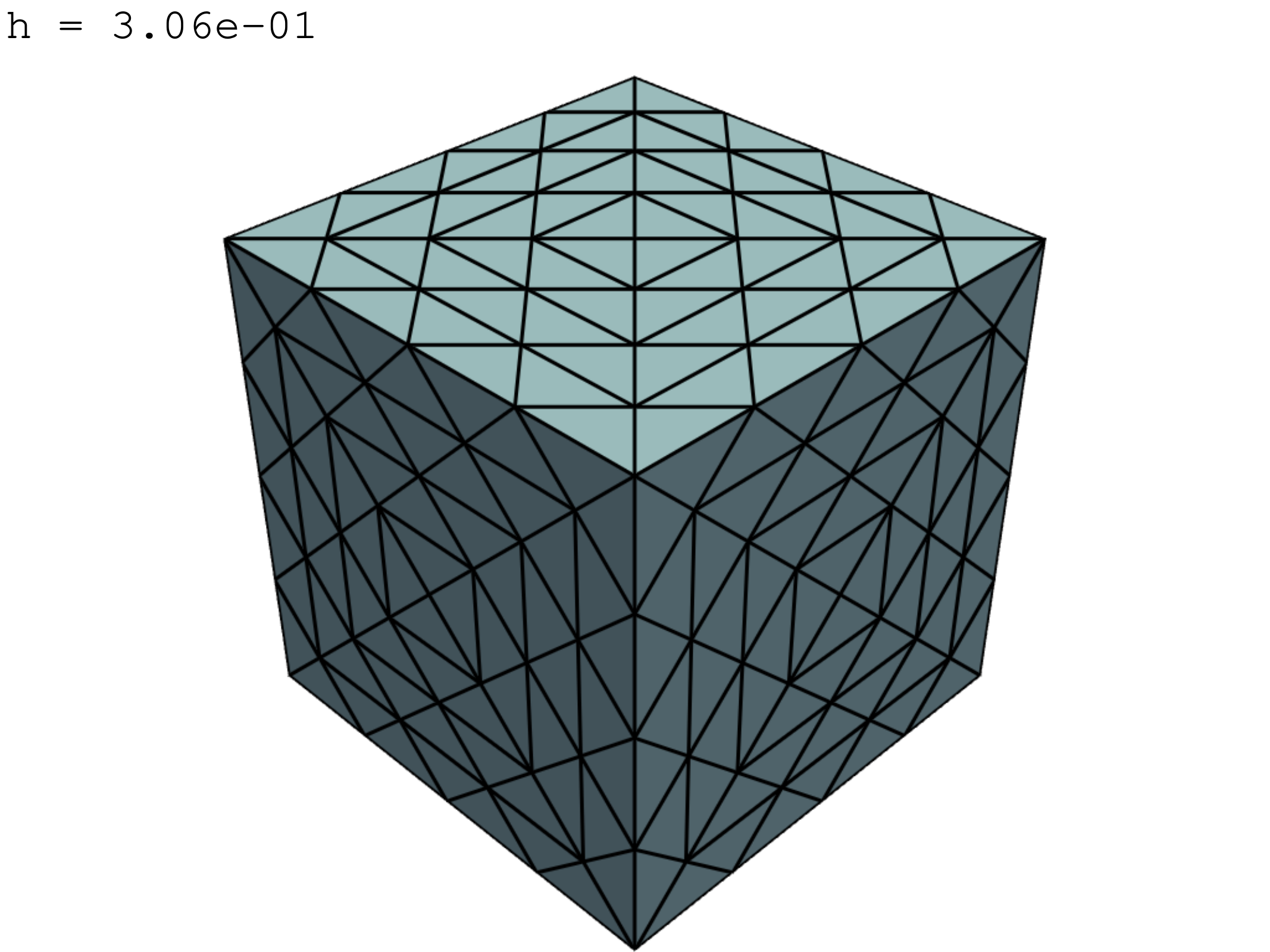}
    \includegraphics[width=0.48\textwidth]{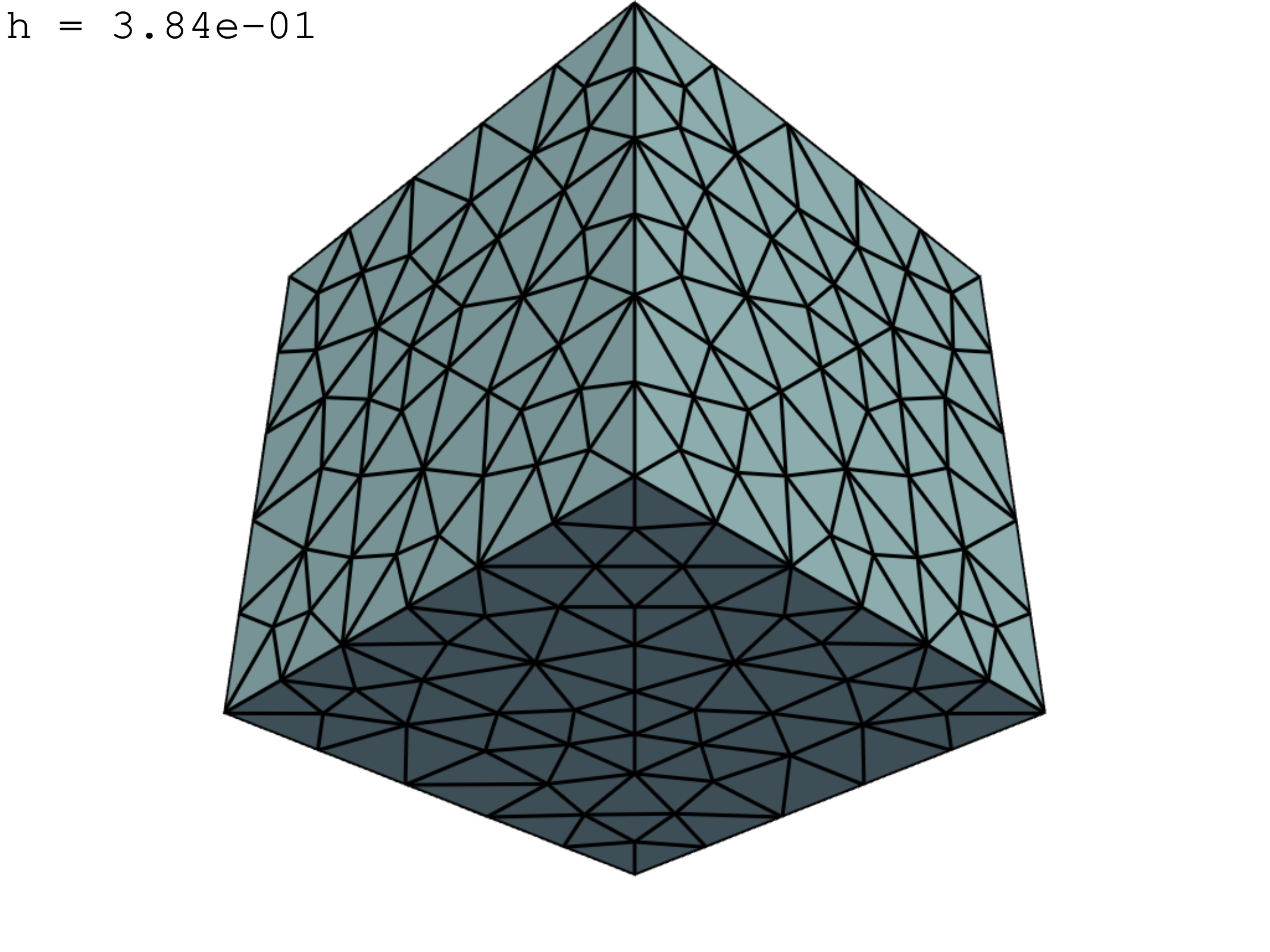}
    \caption{Examples of grid used in the numerical validation. Left: Grid for Section \ref{sec5.1} and \ref{sec5.2}. Right: Grid for Section \ref{sec5.3}.}
    \label{fig5.1}
\end{figure}

It is important to remark that due to the saddle-point structure of the formulation, efficient preconditioning is essential for the numerical performance of the implementation. To this end, we have implemented an operator-based norm-equivalent preconditioning \cite{mardal2011preconditioning}, adapting the approach detailed for the Biot equations \cite{baerland2017weakly}. This preconditioner is robust with respect to the material and discretization parameters considered herein. 

Python scripts were employed to generate all the numerical values and figures in this study. For each numerical example, the table data is generated by a dedicated script. For reproducing numerical results and figures, a Docker image has been created, serialized, and made available for download \cite{DBN_DockerMFECosserat}.

%incorporates the essential computational
%tools, including the three primary components:
%\begin{itemize}
%    \item \texttt{nfempy}: Lightweight library for defining finite element schemes and simulating multiphysics problems.
%    \item \texttt{petsc}: Suite for scaleable solution of scientific applications.
%    \item \texttt{pyvista}: Toolkit for visualization of data in 3D.
%\end{itemize}
%

\subsection{Verification of expected convergence rates and stability for \texorpdfstring{$\ell>0$}{l>0}}\label{sec5.1}
In this section, we validate the stated \emph{a priori} convergence rates of both the strongly and weakly coupled MFE methods detailed in Theorem \ref{th3.9} and \ref{th4.5}. In this subsection, we fix $\lambda_\sigma=1$, while $\ell > 0$ will be considered constant in space. We present results for the two lowest-order spaces, $k=0,1$, for both the weakly and strongly coupled MFE approximations. 

We consider the unit cube domain $\Omega=\left[0,1\right]^3$, with an exact solution $\left(u,r\right)\in C_0^\infty\mathcal{V}$ defined by: 
\begin{subequations}
\begin{align}
    u(x) &= \sum_{i=1}^{3}{\sin{\left(\pi x_i\right)}\left(1-x_{i+1}\right)x_{i+1}\left(1-x_{i-1}\right)x_{i-1}\bm{e}_i}, \label{eq5.2}\\ 
    r(x) &= \sum_{i=1}^{3}{\left(1-x_i\right)x_i\sin{\left(\pi x_{i+1}\right)}\sin{\left(\pi x_{i-1}\right)}\bm{e}_i}, \label{eq5.3}
\end{align}
\end{subequations}
where as usual the indices are understood modulo 3, and $\bm{e}_i=\nabla x_i$ is the unit vector in the $i$-th coordinate direction. The solution is chosen to have high regularity, satisfy zero Dirichlet boundary conditions, and contain a mix of polynomial and geometric terms to reduce the chance of spurious super-convergence phenomena. 

The convergence results for all methods are shown in Figure \ref{fig5.2}, for both $k=0$ and $k=1$. Considering first the results with $\ell=1$, we note that all methods converge according to the expected rates. Moreover, considering the results for $\ell=\left\{{10}^{-2},{10}^{-4}\right\}$, we see that the weakly coupled method is robust for varying $\ell$, while the strongly coupled method SC-RT is as expected not robust, and a large error is observed on coarse grids. The strongly coupled method SC-BDM is also not robust from a theoretical perspective, however, within the range of $\ell$ considered here, it performs well in practice, and no loss of accuracy is seen at any grid level. Keeping in mind that  $A_\omega\sim\ell^{-2}$, this suggests that the SC-BDM method is suitable for a wide range of relative strengths between the stress and couple-stress material tensors. 

We also report improved convergence results in Figure \ref{fig5.3}. For the strongly coupled discretizations, the norms plotted are given in equations \eqref{eq5.IC-norms1}.  For the weakly coupled methods, we report the results for only the error in the projected displacement: 
\begin{align} 
\left\|(\mathcal{p},\mathcal{u})\right\|^*_\mathrm{WC-RT} = 
\left\|(\mathcal{p},\mathcal{u})\right\|^*_\mathrm{WC-BDM} = 
\left\|\varpi_{h,k}u\right\|.
\label{eq5.IC-norms2}
\end{align}

As expected, the weakly coupled methods enjoy good properties of improved convergence, and indeed for the SC-BDM method, all variables enjoy some form of improved convergence, which appears to $\ell$-robust in the range considered herein, so that the method can essentially be considered one degree higher than its nominal rate. The improved convergence of the SC-RT method is also according to expectation, but is not $\ell$-robust in magnitude. The SC-RT method also enjoys an improved convergence in the couple-stress, which is not completely unexpected from literature, but which is known to be difficult to prove in general \cite{bank2019superconvergent}. Finally, our numerical results reveal that the weakly coupled methods only enjoy improved convergence properties with respect to the projected displacement. 

Overall, the observed convergence of the methods is summarized as follows: 

\begin{observation}[Numerical convergence and robustness for spatially constant $\ell$]\label{o5.1} 
    Under the conditions of Assumption \ref{a2.1} and \ref{a2.2a}, both the strongly coupled and weakly coupled MFE methods are convergent, satisfying the rates stated in Theorems \ref{th3.9}, \ref{th3.10}, \ref{thm311} and \ref{th4.5}. For the weakly coupled method, the observed convergence rates are fully robust with respect $\ell$, as expected from theory. For the methods with strong coupling, a sensitivity to small values of $\ell$ is observed, however, the SC-BDM method appears to be more robust than SC-RT in practice. Furthermore, improved convergence rates are observed numerically for additional variables as specified in equations \eqref{eq5.IC-norms1} and \eqref{eq5.IC-norms2}. 
\end{observation}

\begin{figure}[ht!]
    \centering
    \includegraphics[width=0.49\textwidth]{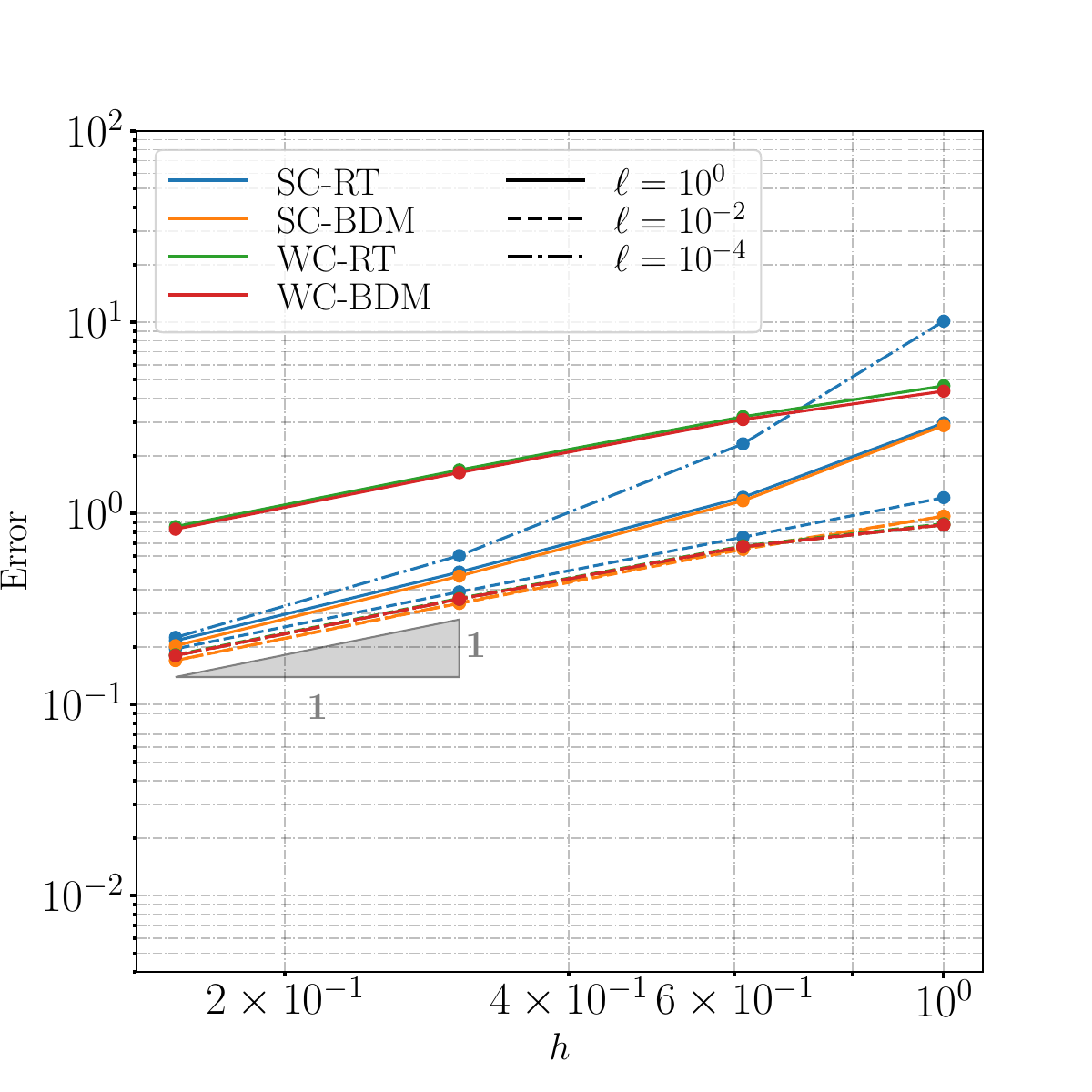}
    \includegraphics[width=0.49\textwidth]{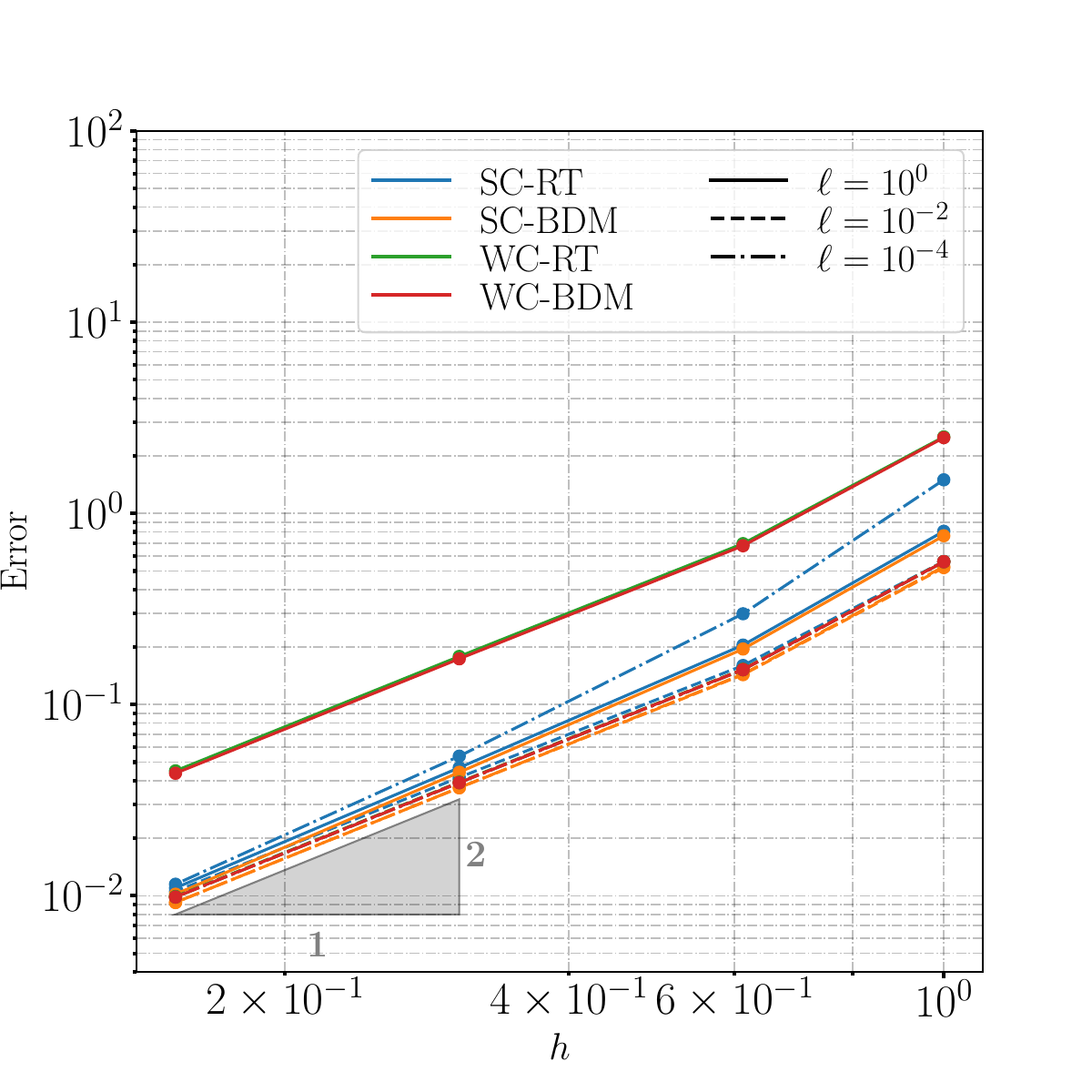}
    \caption{Convergence data for the methods as summarized in Theorem \ref{th3.9} and \ref{th4.5}, for spatially constant $\ell\in\left\{{1,10}^{-2},{10}^{-4}\right\}$ and $\lambda_\sigma=1$. The error is calculated according to equations \eqref{eq3.20} and \eqref{eq4.17a}, for the strongly coupled and weakly coupled methods, respectively. Left panel shows convergence results for $k=0$, right panel for $k=1$.}
    \label{fig5.2}
\end{figure}

\begin{figure}[ht!]
    \centering
    \includegraphics[width=0.49\textwidth]{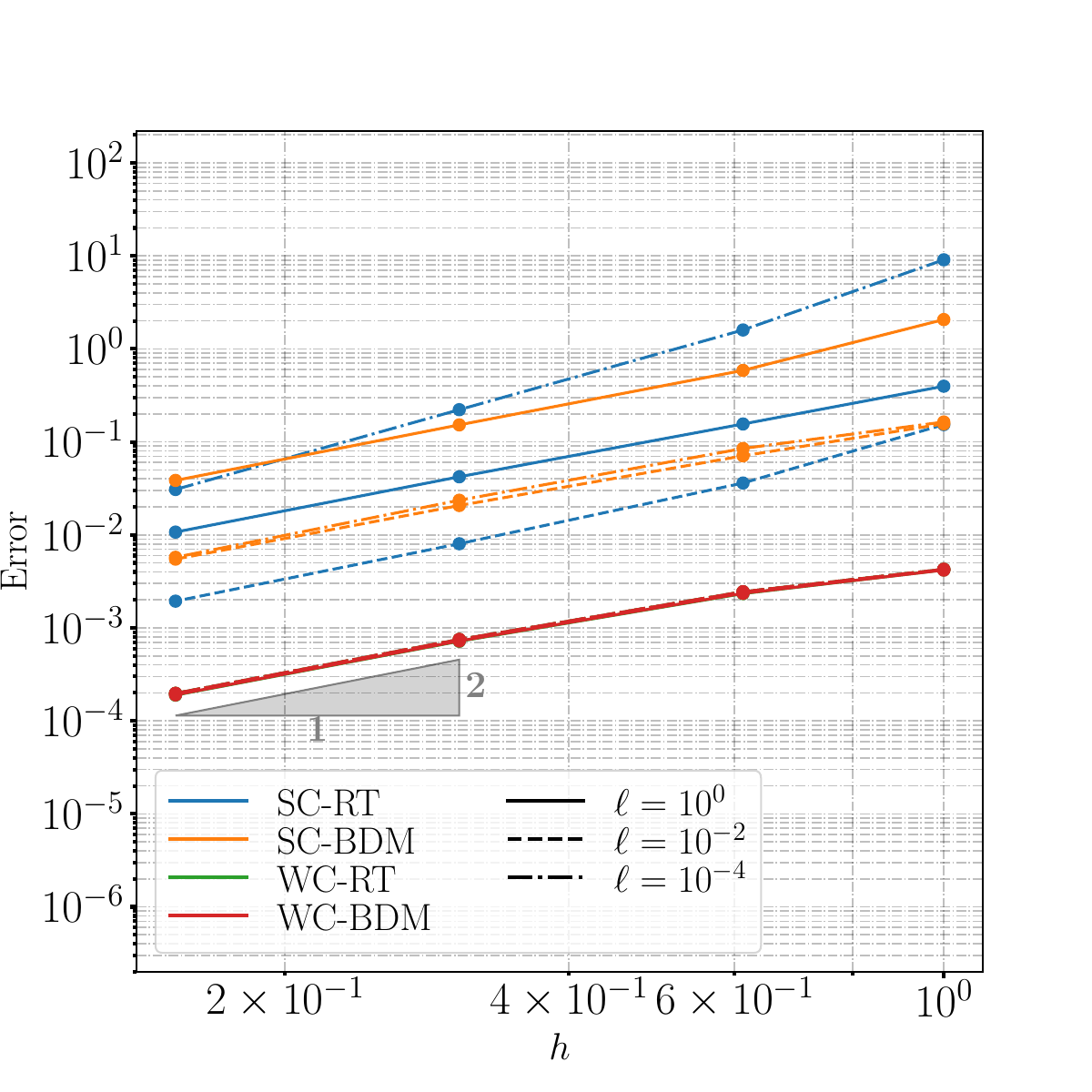}
    \includegraphics[width=0.49\textwidth]{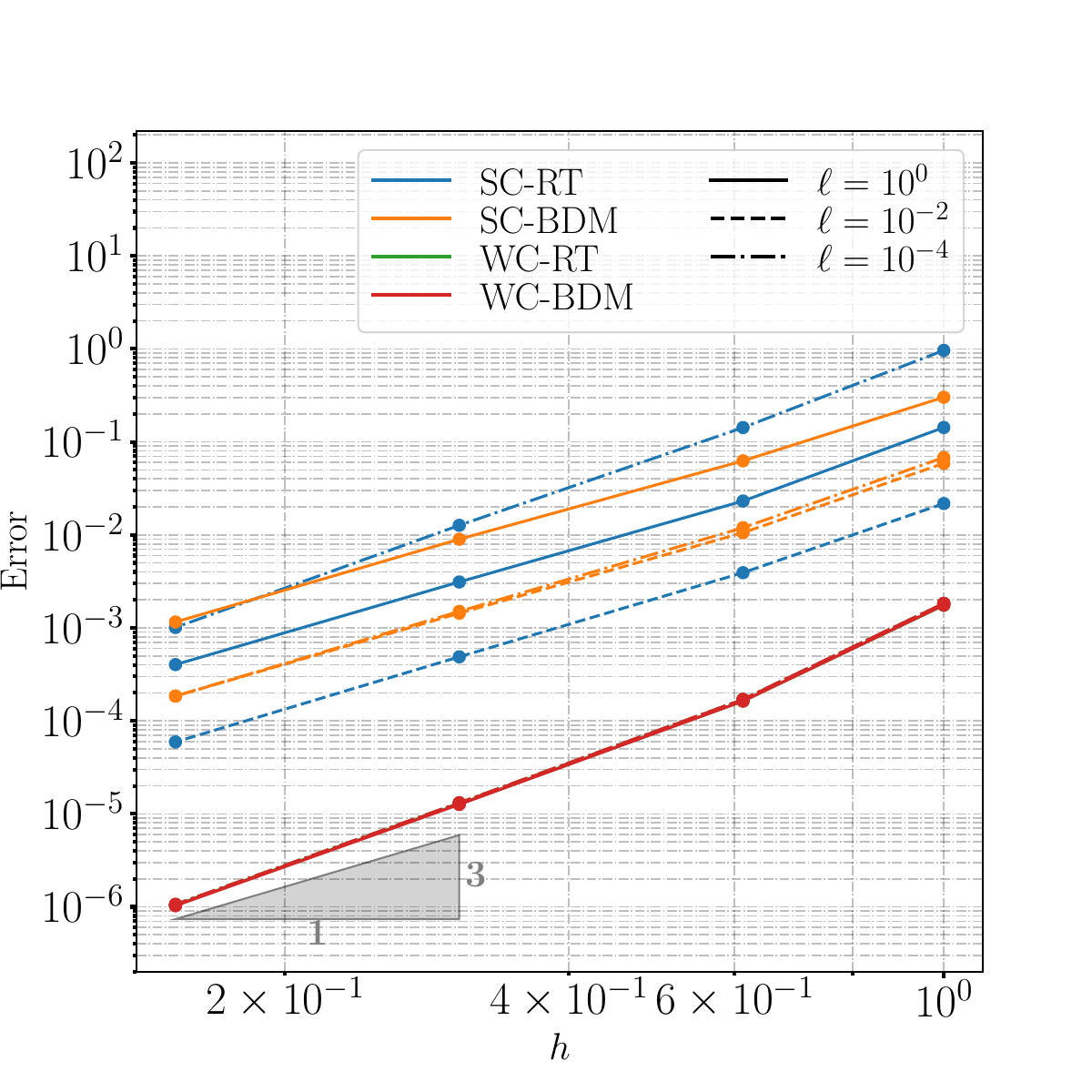}
    \caption{Observed improved convergence rates seen using the norms stated in equations \eqref{eq5.IC-norms1} and \eqref{eq5.IC-norms2} for spatially constant $\ell\in\left\{{1,10}^{-2},{10}^{-4}\right\}$ and $\lambda_\sigma=1$. Left panel shows convergence results for $k=0$, right panel for $k=1$.}
    \label{fig5.3}
\end{figure}

\subsection{Verification of expected convergence rates and stability for nearly incompressible materials}\label{sec5.2}

In this section, we verify the robustness of the strongly and weakly coupled MFE methods for in the limit of nearly incompressible materials, as outlined in Remark \ref{r2.5}. In this subsection, we fix $\ell=1$, while $\lambda_\sigma$ will be considered a parameter constant space. As in the previous section, we present results for the two-lowest-order spaces, $k=0,1$, for both the weakly and strongly coupled MFE approximations. 

We again consider the unit cube domain $\Omega=\left[0,1\right]^3$. However, when considering incompressible materials, it is important to note that the solution will satisfy $\nabla\cdot u=0$ in the limit of $\lambda_\sigma\rightarrow0$. This property does not hold for the expression in \eqref{eq5.2}, therefore we will in this section use the following expression for the exact solution of $u\in C_0^\infty\mathbb{V}$: 
\begin{align}\label{eq5.4}
u(x)=\nabla\times\left(\left(\bm{e}_2+\bm{e}_3\right)\prod_{i=1}^{3}\sin^2{\left(\pi x_i\right)}\right).
\end{align}
Clearly, $u(x)$ defined in \eqref{eq5.4} will satisfy $\nabla\cdot u=0$, and moreover, it is easy to verify that it satisfies the Dirichlet boundary conditions due to the presence of quadratic sine functions.

The results are shown as in the previous section both for convergence in all variables (Figure \ref{fig5.3}). The improved convergence rates (Figure \ref{fig5.4}) are as before evaluated for the strongly coupled discretizations in the norms stated in \eqref{eq5.IC-norms1}. For the weakly coupled methods according, we use the norms: 
\begin{subequations} \label{eq5.IC-norms3}
\begin{align} 
\left\|(\mathcal{p}, \mathcal{u} )\right\|^*_\mathrm{WC-RT} &= \left\|\sigma \right\| + h^{1/2} \|\tilde{\omega}\|_{H(\nabla\cdot\ell)} + \left\|\varpi_{h,k}r\right\|+\left\|\varpi_{h,k}u\right\|,
\label{eq5.IC-norms3a}\\
\left\|(\mathcal{p}, \mathcal{u} )\right\|^*_\mathrm{WC-BDM} &= \left\|\sigma \right\| + \|\tilde{\omega}\| + \left\|\varpi_{h,k}r\right\|+\left\|\varpi_{h,k}u\right\|.
\label{eq5.IC-norms3b}
\end{align}
\end{subequations}% 
Note the scaling of $h^{1/2}$ in front of the second term in the definition of the WC-RT norm. This reflects that we observe somewhere between a half and a full order of improved convergence for this variable.  

We observe that all discretizations are fully robust with respect to variations in $\lambda_\sigma\in\left\{{10}^0,{10}^2,{10}^4\right\}$, and indeed the lines for the different values of $\lambda_\sigma$ can hardly be distinguished in the plots. It is also noteworthy that for non-degenerate $\ell$, the weakly coupled discretizations also enjoys almost as favorable convergence properties as their strongly coupled counterparts. 

\begin{observation}[Robustness for incompressible materials]\label{o5.2} 
    Under the conditions of Assumption \ref{a2.1} and \ref{a2.2a}, both the strongly coupled and weakly coupled MFE methods are convergent, satisfying the rates stated in Theorems \ref{th3.9}, \ref{th3.10}, \ref{thm311} and \ref{th4.5}. Additional improved convergence rates are observed according to equations \eqref{eq5.IC-norms1} and \eqref{eq5.IC-norms3}. All results  are fully robust with respect to $\lambda_\sigma$, as expected from theory.   
\end{observation}

\begin{figure}[ht!]
    \centering
    \includegraphics[width=0.49\textwidth]{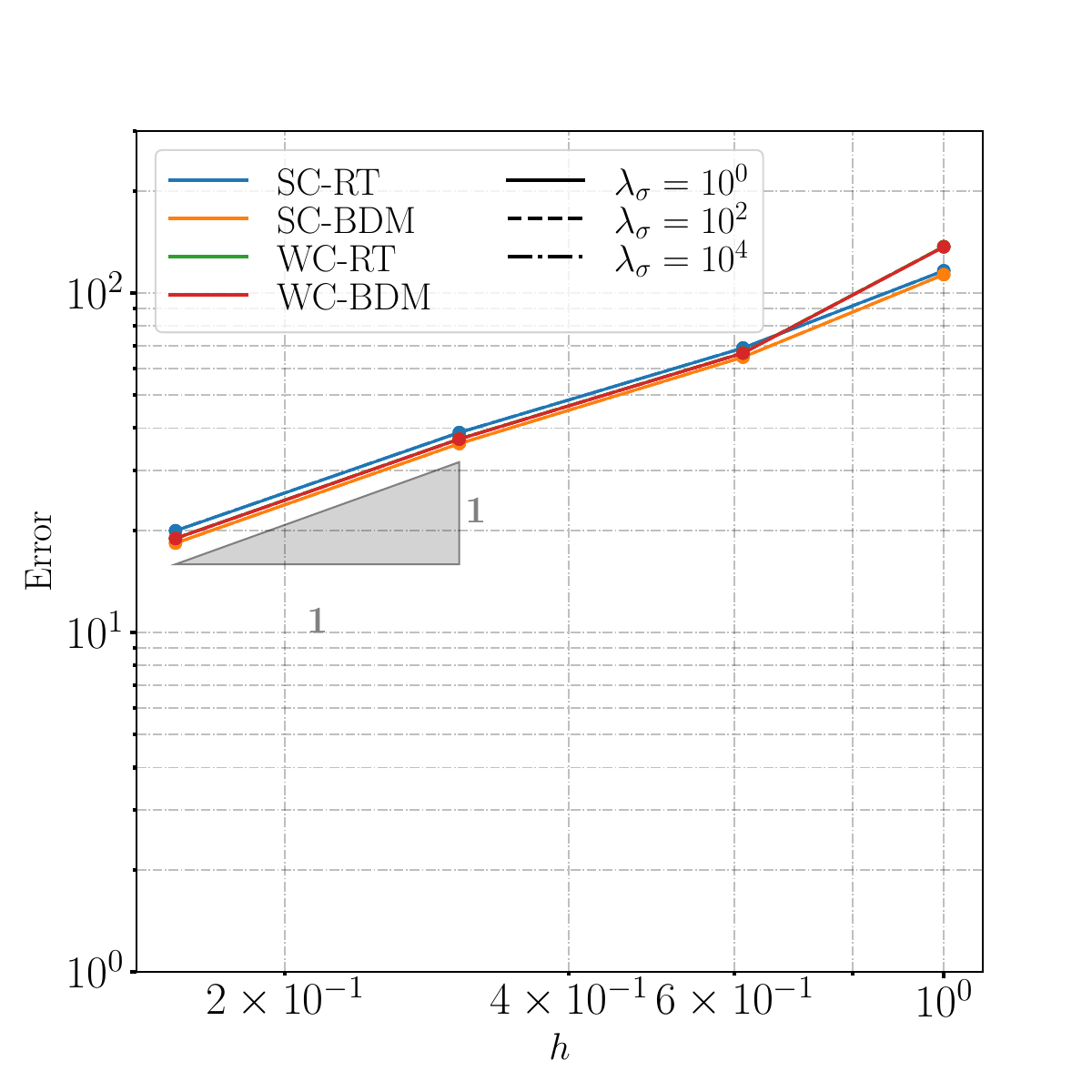}
    \includegraphics[width=0.49\textwidth]{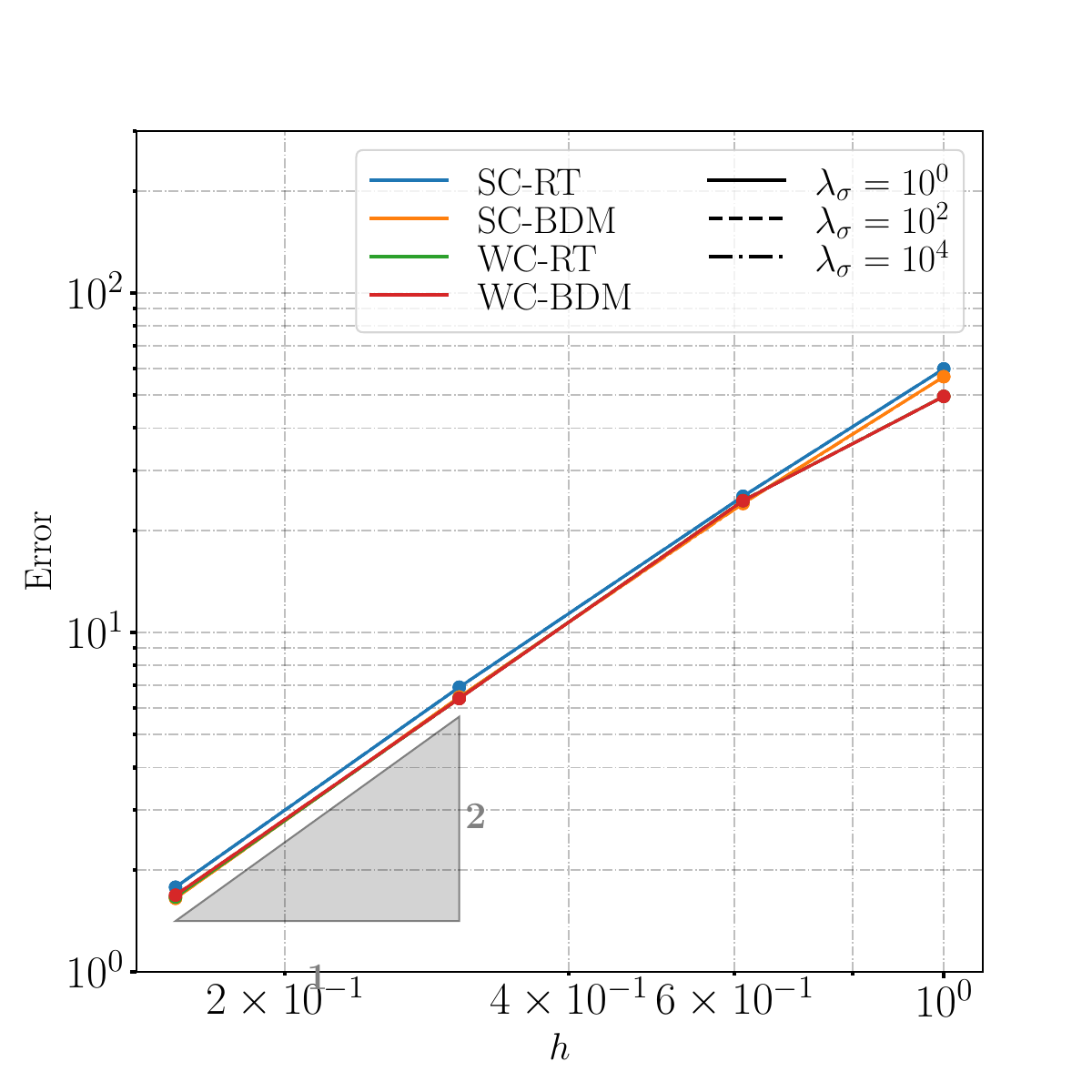}
    \caption{Convergence data for the methods as summarized in Theorem \ref{th3.9} and \ref{th4.5}, for spatially constant $\lambda_\sigma\in\left\{{1,10}^2,{10}^4\right\}$ and $\ell=1$. The error is calculated according to equations \eqref{eq3.20} and \eqref{eq4.17a}, for the strongly coupled and weakly coupled methods, respectively. Left panel shows convergence results for $k=0$, right panel for $k=1$. Note that as the methods are fully robust with respect to $\lambda_\sigma$, the lines are overlapping within the resolution of the figure.}
    \label{fig5.4}
\end{figure}

\begin{figure}[ht!]
    \centering
    \includegraphics[width=0.49\textwidth]{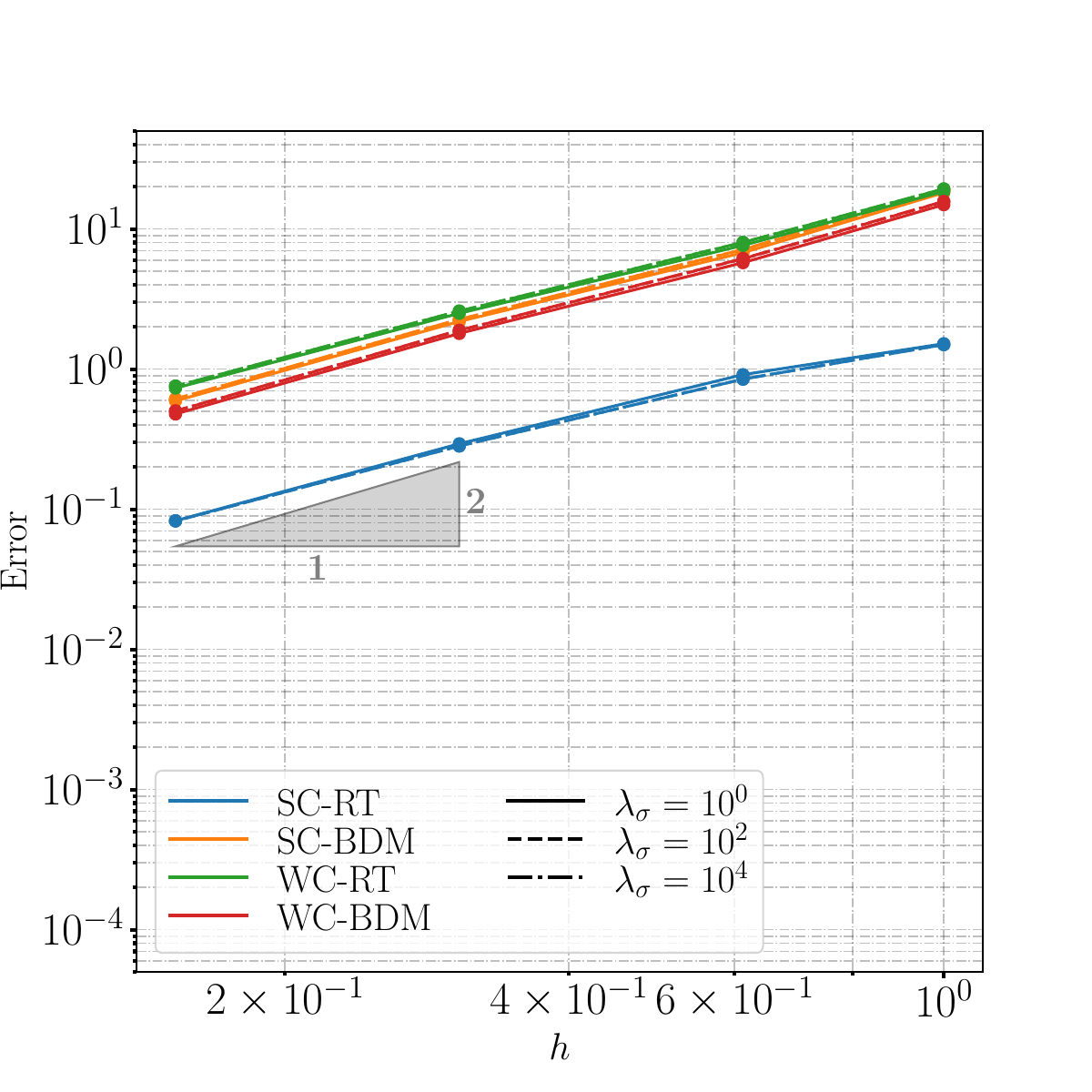}
    \includegraphics[width=0.49\textwidth]{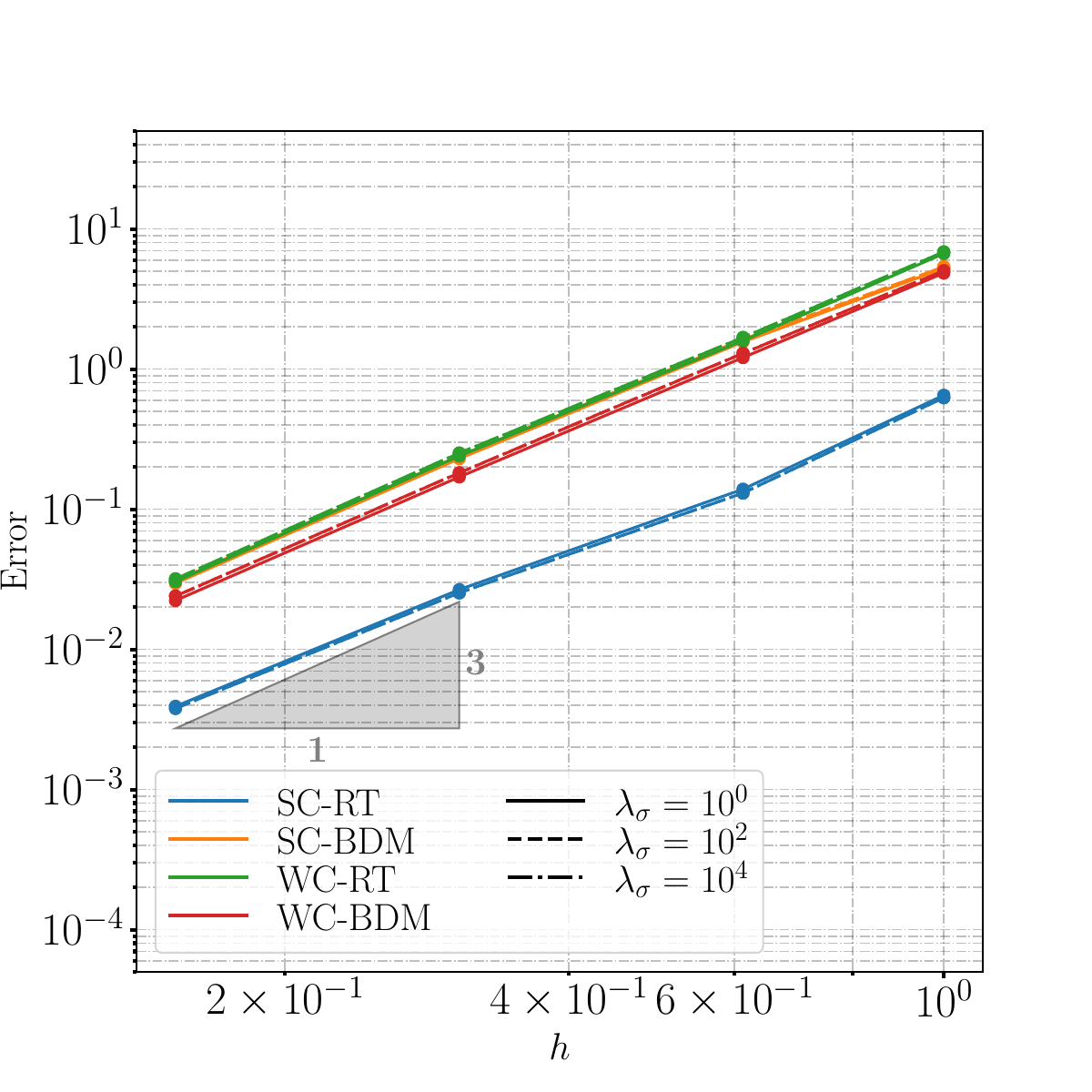}
    \caption{Super-convergence data for the strongly coupled methods as summarized in Theorem \ref{th3.9}. for spatially constant $\lambda_\sigma\in\left\{{1,10}^2,{10}^4\right\}$ and $\ell=1$. Super-convergence appears in the rotation and couple-stress variables, and is calculated according to \eqref{eq3.27}. Left panel shows convergence results for $k=0$, right panel for $k=1$. Note that as the methods are fully robust with respect to $\lambda_\sigma$, the lines are overlapping within the resolution of the figure.}
    \label{fig5.5}
\end{figure}

\subsection{Verification of robustness for degenerate \texorpdfstring{$\ell$}{l}}\label{sec5.3}
Theorem \ref{th4.5} ensures that the stability and approximation properties of the weakly coupled MFE discretization is preserved in the presence of problems where regions contain fully degenerate Cosserat materials, i.e. regions of the domain where a standard linear elastic model is employed. In this final example, we therefore return to the same domain and manufactured solution as in Section \ref{sec5.1}, with $\lambda_\sigma=1$. However, in this section we consider a spatially varying $\ell(x)$, with a nontrivial region where $\ell(x)=0$. This thereby allows us to showcase the applicability of the weakly coupled MFE discretization to such composite materials, where regions of the material are a regular elastic material, and simultaneously verify the stability of the method in the presence of variability of $\ell$, including the region where $\ell=0$.

Concretely, we define the following spatially variable $\ell(x)$: 
\begin{align}\label{eq5.5}
\ell(x)=\min{\left(1,\max{\left(0,\max_{i=1\ldots3}{\left(3x_i-1\right)}\right)}\right)}.
\end{align}
See Figure \ref{fig5.6} for a 2D illustration of $\ell$ restricted to any plane defined by $x_3\in\left[0,1/3\ \right]$ (left), and for a view of the solution of the couple stress variable (right). We note that by this construction, $\ell\in C^0$ and $0\le\ell\le1$. Moreover, $\ell$ is piecewise linear with $| \nabla \ell | \le 3$ almost everywhere and thereby satisfies the conditions of Assumption \ref{a2.2b}. Note also that $\ell=0$ in the cube close to the origin defined by $\max_{i=1\ldots3}{\left(x_i\right)}\le1/3$, thus in this region the constitutive model corresponds to linear elasticity. 

The strongly coupled discretizations are not applicable to this problem, as $\|A_\omega\|$ is unbounded. Correspondingly, we only consider the weakly coupled discretizations of Theorem \ref{th4.5}, with the results being presented in Figure \ref{fig5.7}. 
The error of the WC-RT and WC-BDM methods are indistinguishable in the plot, as the couple stress error is (as expected) not dominating for this problem. 
We observe that the presence of a degenerate region does not affect the convergence of the method, and the theoretically expected rates are  achieved. It is important to emphasize that this result reflects the careful definition of norms, since in the elastic region the regularity of the rotations will in general only be $r\in L^2\mathbb{V}$. We summarize this as follows.

\begin{observation}[Robustness in transition to linear elasticity]\label{o5.3} 
    Under the conditions of Assumption \ref{a2.1} and \ref{a2.2b}, the weakly coupled MFE discretizations are convergent and stable in the presence of degenerate regions of the domain where the Cosserat model reduces to linear elasticity. 
\end{observation}

\begin{figure}[ht!]
    \centering
    \includegraphics[width=0.37\textwidth]{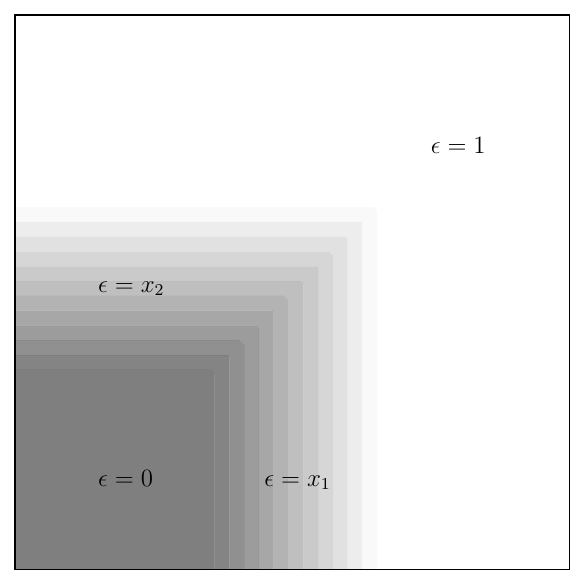}
    \includegraphics[width=0.5\textwidth]{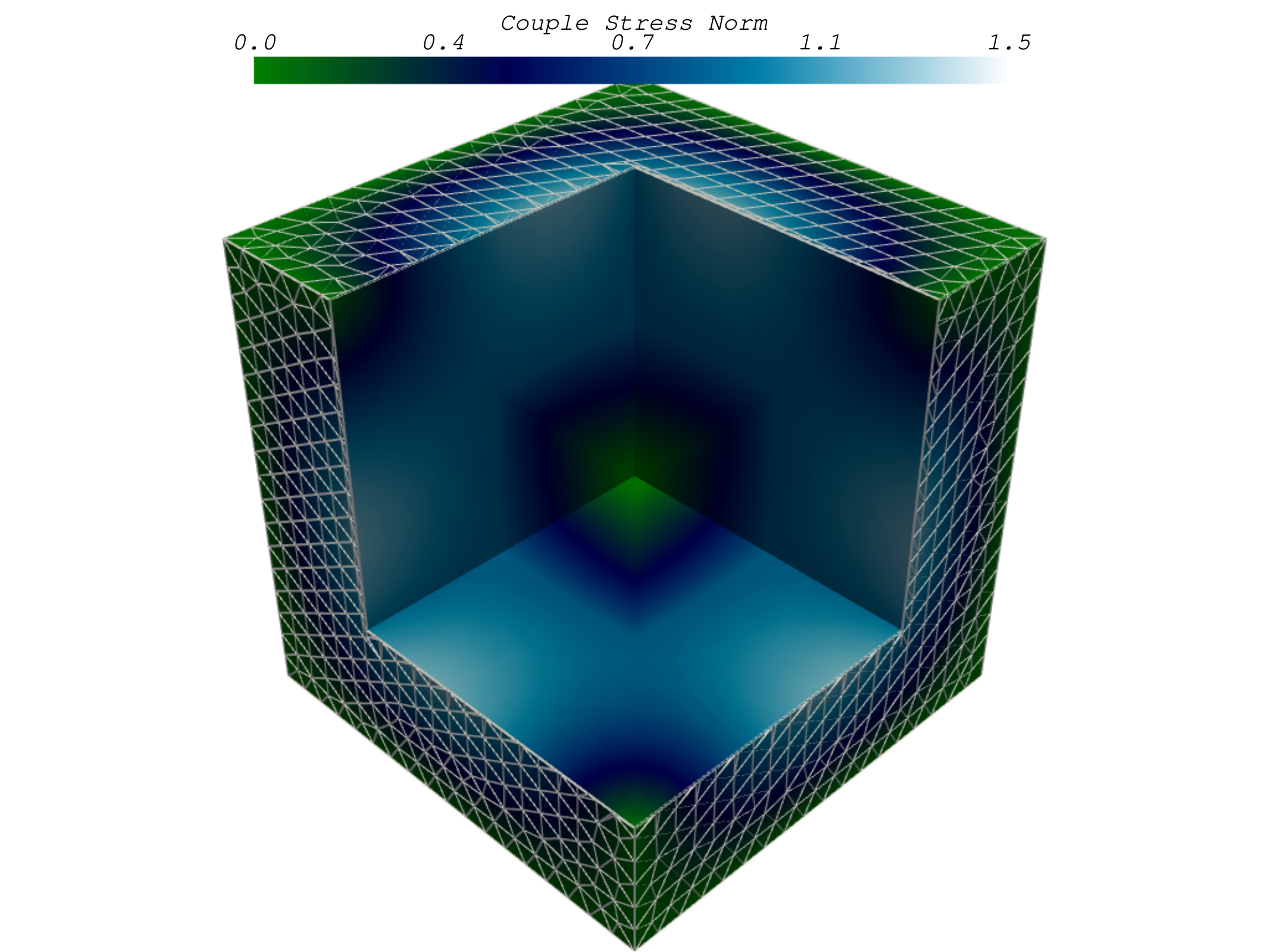}
    \caption{Left: Plot of spatially varying $\ell$ used in Section \ref{sec5.3}, as restricted to a plane defined by e.g. $x_i=0$.  Right: Plot of the (numerical) solution for the couple stress variable, seen facing the origin along the $x_1=x_2=x_3$ line. The blue square in the center of the cut-out indicates that the ${\tilde{\omega}}_h=0$ in the region of pure elasticity.}
    \label{fig5.6}
\end{figure}
 
\begin{figure}[ht!]
    \centering
    \includegraphics[width=0.49\textwidth]{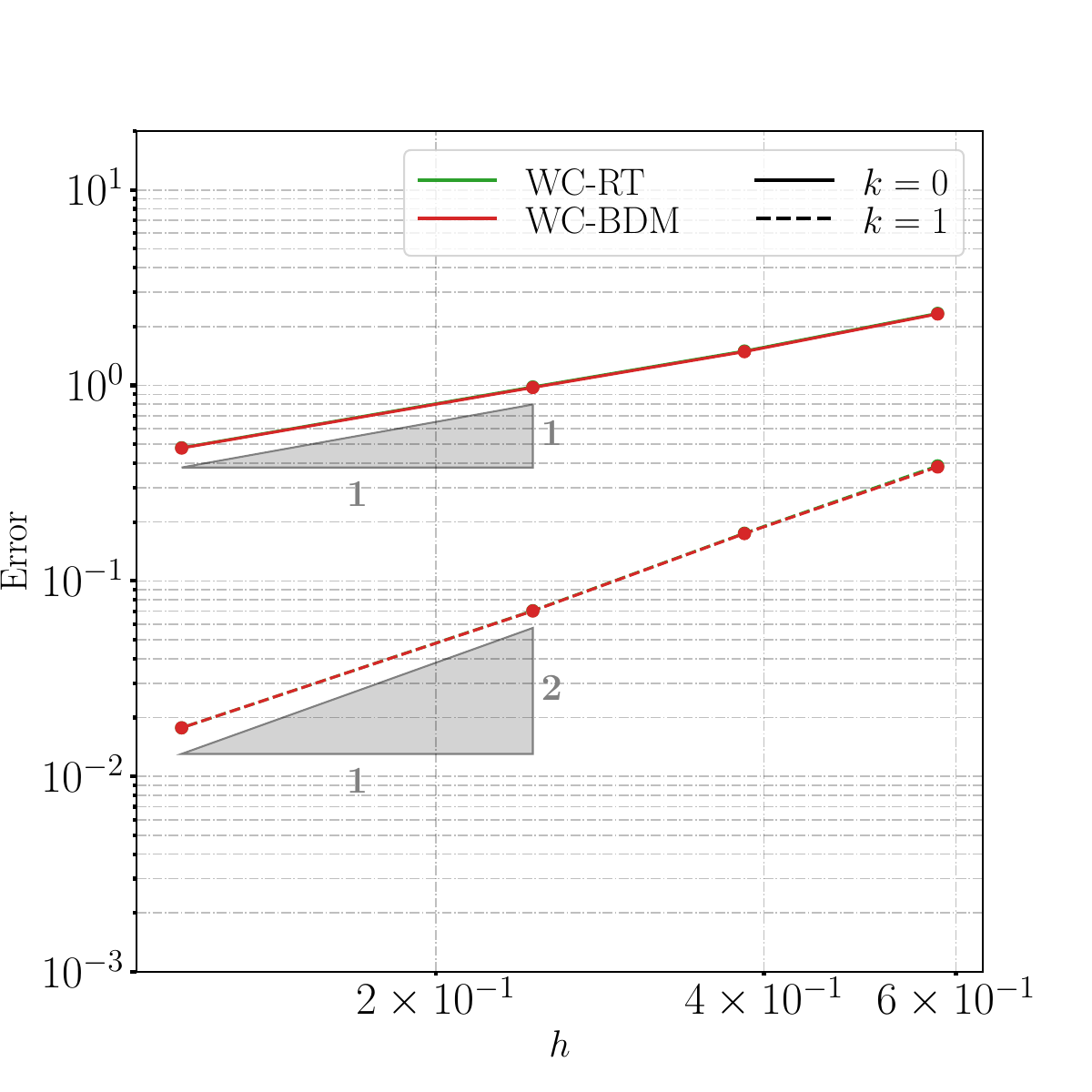}
    \includegraphics[width=0.49\textwidth]{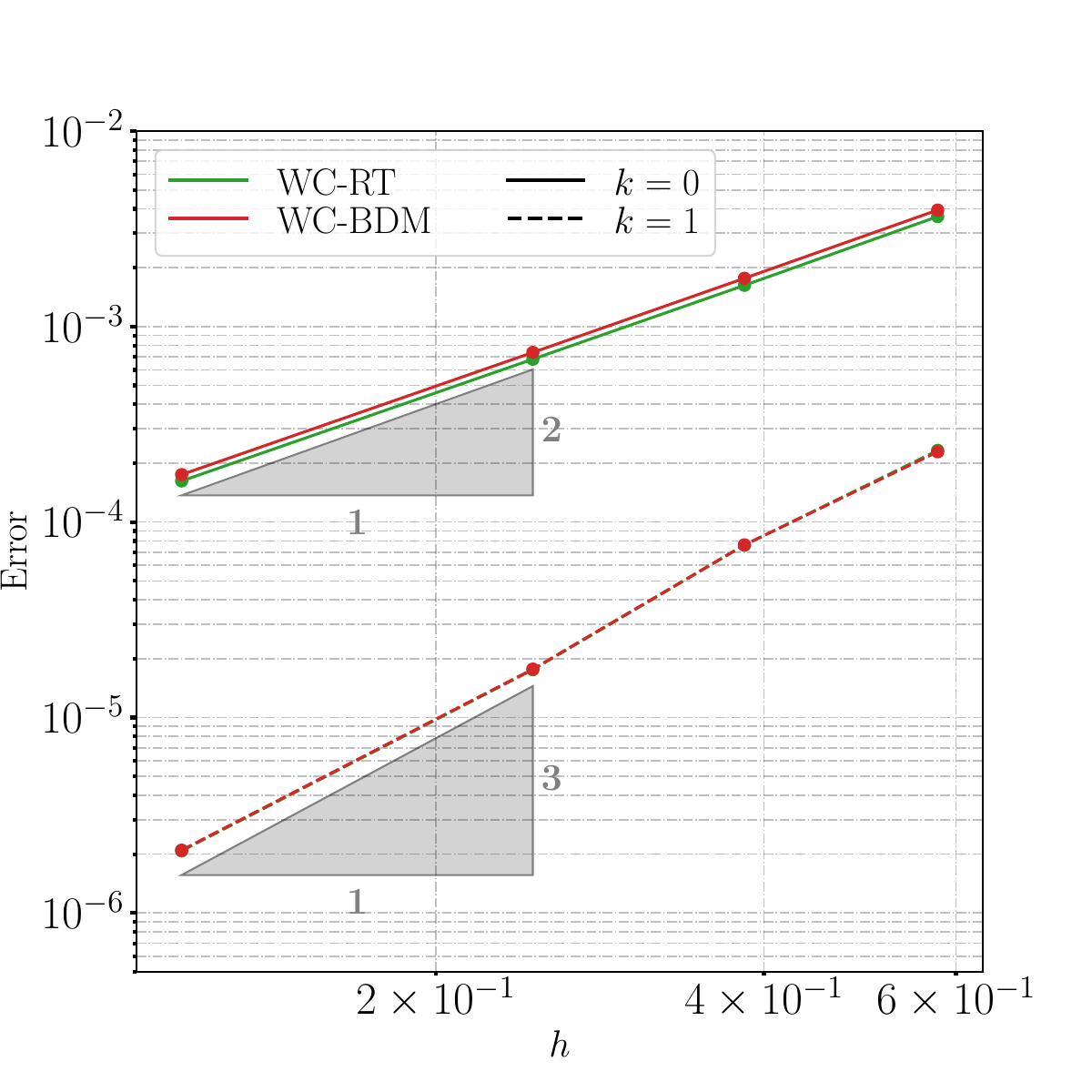}
    \caption{Left: Convergence data for the weakly coupled discretizations of Theorem \ref{th4.5} for spatially varying $\ell$, and $k\in\left\{0,1\right\}$. The error is calculated according to equations \eqref{eq4.17a}. Right: Results showing improved convergence rates according to the norms defined in equation \eqref{eq5.IC-norms3}.}
    \label{fig5.7}
\end{figure}

\section*{Reproducible data and figures}
 All data and figures presented herein can be reproduced using a Docker image that has been created and made publicly available \cite{DBN_DockerMFECosserat}. Source code is included in the Docker image for use, modification, and redistribution.

\section*{Acknowledgments}
The authors acknowledge the impact on this work of Inga Berre and Eirik Keilegavelen through many discussions on this and related topics. Omar Duran was supported by the European Research Council (ERC), under the European Union’s Horizon 2020 research and innovation program (grant agreement No 101002507).

\bibliographystyle{siamplain}
\bibliography{ex_article}
\end{document}